\documentclass{svjour3}%
\usepackage{amssymb}
\usepackage{amsfonts}
\usepackage{amsmath}
\usepackage[numbers]{natbib}
\usepackage{graphicx}%
\providecommand{\U}[1]{\protect\rule{.1in}{.1in}}
\newtheorem{criterion}[theorem]{Criterion}
\begin{document}

\journalname{Mathematics of Control, Signals, and Systems}
\title{Algebraic Invariance Conditions in the Study of Approximate
(Null-)Controllability of Markov Switch Processes }
\author{Dan Goreac
\and Miguel Martinez}
\institute{Universit\'{e} Paris-Est, LAMA (UMR 8050), UPEMLV, UPEC, CNRS, 77454
Marne-la-Vall\'{e}e, France Corresponding author, Email :
Dan.Goreac@univ-mlv.fr, Tel. : +33 (0)1 60 95 75 27, Fax : +33 (0)1 60 95 75
45, Email : Miguel.Martinez@univ-mlv.fr}
\maketitle

\begin{abstract}
We aim at studying approximate null-controllability properties of a particular
class of piecewise linear Markov processes (Markovian switch systems). The
criteria are given in terms of algebraic invariance and are easily computable.
We propose several necessary conditions and a sufficient one. The hierarchy
between these conditions is studied via suitable counterexamples. Equivalence
criteria are given in abstract form for general dynamics and algebraic form
for systems with constant coefficients or continuous switching. The problem is
motivated by the study of lysis phenomena in biological organisms and price
prediction on spike-driven commodities.

\end{abstract}

\section{Introduction}

This paper focuses on a particular class of controlled piecewise deterministic
Markov processes (PDMP) introduced in \cite{Davis_84}, see also
\cite{davis_93}. Namely, we are interested in the approximate (null-)
controllability of switch processes. A switch process is often used to model
various aspects in biology (see \cite{cook_gerber_tapscott_98},
\cite{crudu_debussche_radulescu_09}, \cite{Wainrib_Thieullen_Pakdaman_2010},
\cite{crudu_Debussche_Muller_Radulescu_2012}, \cite{G8}), reliability or
storage modelling (in \cite{Boxma_Kaspi_Kella_Perry_2005}), finance (in
\cite{Rolski_Schmidli_2009}), communication networks (\cite{graham2009}), etc.
We can describe these processes as having two components denoted $\left(
\gamma,X\right)  $. In our framework, the mode component $\gamma$ evolves as a
pure jump Markov process and cannot be controlled. It corresponds to spikes
inducing regime switching. The second component $X$ obeys a controlled linear
stochastic differential equation (SDE) with respect to the compensated random
measure associated to $\gamma$. The linear coefficients of this controlled SDE
also depend on the current mode $\gamma.$ We are concerned with a problem
fitting the framework of controllability for stochastic jump systems. The main
aim of the paper is to exhibit explicit algebraic conditions on the linear
coefficients under which, given a time horizon $T>0,$ for every initial
configuration $\left(  \gamma_{0},x_{0}\right)  $ the $X_{T}$ component can be
steered arbitrarily close to acceptable targets (random variables, adapted to
the filtration generated by the random measure).

Let us briefly explain the applications we have in mind. The first class of
applications comes from the theory of genetic applets. The simplest model is
the one introduced in \cite{cook_gerber_tapscott_98} in connection to gene
expression. The model can be resumed by the following diagram.%
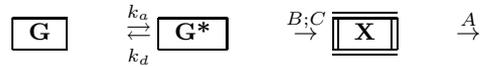
\[%
\begin{tabular}
[c]{|l|}\hline
\textbf{G}\\\hline
\end{tabular}
\ \ \ \ \ \ \underset{k_{d}}{\overset{k_{a}}{\rightleftarrows}}%
\begin{tabular}
[c]{|l|}\hline
\textbf{G*}\\\hline
\end{tabular}
\ \ \ \ \ \ \overset{B;C}{\rightarrow}%
\begin{tabular}
[c]{||l||}\hline\hline
\textbf{X}\\\hline\hline
\end{tabular}
\ \ \ \ \ \ \overset{A}{\rightarrow}%
\]
It reduces the gene expression phenomenon to a single gene which toggles
between an inactive state $G$ and an active one $G^{\ast}$ and to a set of
proteins $X.$ Activation and deactivation occur at different rates $k_{a}$ and
$k_{d}$. Thus, the mode component $\gamma$ is given by a random process
toggling between active and inactive at specific rates. When active, the gene
can express (some of) the products $X.$ We can have a continuous expression at
some rate (matrix) $B$ per time unit or a unique burst rate proportional to
the protein concentration. The products degrade at some rate $A.$ Various
external factors (e.g. catalyzers, temperature, generically denoted by $u$)
can lead to the enhancement of continuous expression and a first order
approximation would lead to an expression term $B(\gamma)u$ ($Bu$ if the gene
is active, $0$ otherwise)$.$ The classical method in analytical chemistry
would lead to considering a piecewise differential behavior for the proteins
concentration $\overset{\cdot}{X}_{t}=-A\left(  \gamma_{t}\right)
X_{t}+B(\gamma_{t})u_{t}$ and, eventually a jump of intensity $C\left(
\gamma_{t-}\right)  X_{t-},$ hence leading to a switch model. In this kind of
simple model, the proteins act as their own regulator and a null concentration
of proteins leads to the death of the system. Our method gives simple
algebraic conditions on the model ($A\left(  \gamma\right)  ,$ $B\left(
\gamma\right)  $ and $C\left(  \gamma\right)  $) under which the system can be
controlled from given protein concentrations to cellular death in a given
finite time horizon.

A second motivation comes from mathematical finance. Let us suppose that a
controller needs to "predict" a vector price of some derivatives on some
commodity (say energy) starting from previous predictions $x_{0}$ (whose
precision is unknown). The fluctuations in the commodity are observed as
spikes (the $\gamma$ component) leading to a pure jump process. As long as the
commodity does not change, the price process $X$ is imposed to have a linear
behavior (induced by the interest rate, for example) proportional to some
$A\left(  \gamma\right)  $ matrix. The controller following the market
information imposes a continuous linear correction term ($B$). When a spike is
noticed, the price $X$ is multiplied by some coefficient ($I+C$)$.$ Given a
finite time horizon $T$, the actual price at time $T$ can be any random
variable with respect to the information given by the commodity at time $T.$
The idea is to give explicit conditions on the contract terms $A,B,C$ under
which, independently of the initial prediction $x_{0},$ the controller is able
to reach the actual price (or, to be arbitrarily close to it).

In the finite-dimensional deterministic setting, exact controllability of
linear systems has been characterized by the so-called Kalman criterion.
Alternatively, one studies the dual notion of observability via Hautus' test
as in \cite{Hautus}. These conditions can also be written in terms of
algebraic invariance of appropriate linear subspaces and the notions easily
adapt to infinite dimensional settings in \cite{Schmidt_Stern_80},
\cite{Curtain_86}, \cite{russell_Weiss_1994}, \cite{Jacob_Zwart_2001},
\cite{Jacob_Partington_2006}, etc.

In the stochastic setting, a duality approach to the different notions of
controllability would lead to backward stochastic differential equations (in
short BSDE introduced in \cite{Pardoux_Peng_90}). This method allows one to
characterize the exact (terminal-) controllability of Brownian-driven control
systems via a full rank condition (see \cite{Peng_94}). Whenever this
condition is not satisfied, one characterizes the approximate controllability
and approximate null-controllability using invariance-like criteria (see
\cite{Buckdahn_Quincampoix_Tessitore_2006} for the control-free noise setting
and \cite{G17} for the general Brownian setting). In the case when the
stochastic perturbation is of jump-type, exact controllability cannot be
achieved (as consequence of the incompleteness; see \cite{Merton_76}). Hence,
the "good" notions of controllability are approximate (resp. approximate
null-) ones. For Brownian-driven control systems, the two approximate
controllability notions are equivalent (cf. \cite{G17}). This is no longer the
case if the system has an infinite-dimensional component (see \cite{G1}
treating mean-field Brownian-driven systems). Various methods can be employed
in infinite-dimensional state space Brownian setting leading to partial
results (see \cite{Fernandez_Cara_Garrido_atienza_99},
\cite{Sarbu_Tessitore_2001}, \cite{Barbu_Rascanu_Tessitore_2003}, \cite{G16}).
For jump systems, BSDE have first been considered in
\cite{BarlesBuckdahnPardoux} in a Brownian and Poisson random measure setting.
Using this tool, the approach of \cite{Buckdahn_Quincampoix_Tessitore_2006}
generalizes to a L\'{e}vy setting (see \cite{G10}). BSDE driven by random
measures are studied in \cite{Xia_00}, while backward systems driven by marked
point processes make the object of \cite{Confortola_Fuhrman_2013},
\cite{Confortola_Fuhrman_2014}, \cite{Confortola_Fuhrman_Jacod_2014}.

In this paper, we are interested in the study of approximate controllability
properties for a particular class of piecewise deterministic Markov processes
(introduced in all generality in \cite{Davis_84}), namely Markovian systems of
switch type. We propose explicit algebraic invariance conditions which are
necessary or sufficient for the approximate (null-) controllability of this
class of jump systems. We emphasize that these algebraic conditions are easily
computable. We propose several examples and counterexamples illustrating the
necessary and sufficient conditions in all generality. Explicit equivalent
criteria are obtained for two particular classes of systems having either
constant coefficients or continuous switching. The approach relies on duality,
backward stochastic differential equations techniques and Riccati systems.
Although, in all generality, only partial conditions are given, the main
theoretical contribution of the paper is that it considers non-constant
coefficients and a framework in which the jump intensity depends on the trajectory.

In Section \ref{SectionGeneralSwitch} we recall some elements on Markov
pure-jump processes governing the mode component and piecewise-linear switch
processes. Using the controllability operators and the adjoint BSDE, we give a
first (abstract) criterion for approximate and approximate
null-controllability in Theorem \ref{TheoremCar}. The study of algebraic
invariance characterizations for the approximate null-controllability forms
the objective of Section \ref{SectionExplicitInvariance}. The previously
proven abstract tool allows us to obtain two necessary conditions for
approximate null-controllability in Propositions \ref{PropNec1} and
\ref{PropNec2} as well as a sufficient criterion in Proposition \ref{PropSuf1}%
. These conditions are explicit and involve invariance or strict invariance
algebraic notions. By means of counterexamples in the framework of bimodal
switch systems, we show in Section \ref{SectionExpEquiv} that no hierarchy can
be established between these necessary conditions and/or the Kalman criterion
for the control of the associate deterministic system. We exhibit a
two-dimensional example in which the necessary condition of Proposition
\ref{PropNec1} fails to imply the deterministic Kalman condition and, hence,
may not be sufficient for general controllability (see Example
\ref{Nec1notDetCtrl}). Even if both the necessary condition of Proposition
\ref{PropNec1} and Kalman deterministic condition are satisfied, the system
may fail to satisfy the condition of Proposition \ref{PropNec2}, and, thus,
might not be approximately null-controllable (see Example
\ref{Nec1+DetCtrlnotNec2}). In a similar way, the necessary condition of
Proposition \ref{PropNec2} and Kalman deterministic condition do not (in all
generality) imply the condition of Proposition \ref{PropNec1} (see Example
\ref{Nec2+DetCtrlnotNec1}). In the framework of systems with constant
coefficients (Subsection \ref{SubsectionConstantCoeff}), we prove that the
necessary condition of Proposition \ref{PropNec2} is also sufficient. The
basic idea is to give a deterministic characterization for the local viability
kernel of $KerB^{\ast}$ by means of linear-quadratic control techniques in
Proposition \ref{PropViablocKerB*} which allows to weaken the sufficient
condition (see Criterion \ref{CritEquiv}). This is made possible by Riccati
techniques which, in this particular setting, concern deterministic equations
(instead of backward Riccati stochastic systems). We prove that for this class
of systems, the necessary and sufficient condition (given in Criterion
\ref{CritEquiv}) can be weaker than the sufficient one in Proposition
\ref{PropSuf1} (see Example \ref{CtrlnotSuf1}). The second class for which we
can achieve complete description of approximate null-controllability is that
of continuous linear systems of switch type. This forms the objective of
Subsection \ref{SubsectionContSwitching}. The proof of the equivalence
Criterion \ref{CritContSwitch} can be easily extended to more general systems
in which the mode is equally influenced by the continuous component. It relies
on explicit construction of stabilizing open-loop control processes. Finally,
we give a hint on work to come in Section \ref{SectionConclusionsPerspectives}.

\section{A Coefficient-Switch Markov Model\label{SectionGeneralSwitch}}

\subsection{Markov Jump Processes\label{SectionConstruction}}

We briefly recall the construction of a particular class of Markov pure jump,
non explosive processes on a space $\Omega$ and taking their values in a
metric space $\left(  E,\mathcal{B}\left(  E\right)  \right)  .$ Here,
$\mathcal{B}\left(  E\right)  $ denotes the Borel $\sigma$-field of $E.$ The
elements of the space $E$ are referred to as modes. These elements can be
found in \cite{davis_93} in the particular case of piecewise deterministic
Markov processes; see also \cite{Bremaud_1981}. To simplify the arguments, we
assume that $E\subset%
\mathbb{R}
^{m}$, for some $m\geq1$. The process is completely described by a couple
$\left(  \lambda,Q\right)  $

(i) a Lipschitz continuous jump rate $\lambda:E\longrightarrow%
\mathbb{R}
_{+}$ such that $\underset{\theta\in E}{\sup}\left\vert \lambda\left(
\theta\right)  \right\vert \leq c_{0}$ and

(ii) a transition measure $Q:E\longrightarrow\mathcal{P}\left(  E\right)  $,
where $\mathcal{P}\left(  E\right)  $ stands for the set of probability
measures on $\left(  E,\mathcal{B}\left(  E\right)  \right)  $ such that
\[
\left.
\begin{array}
[c]{l}%
(ii1)\text{ }Q\left(  \gamma,\left\{  \gamma\right\}  \right)  =0,\\
(ii2)\text{ for each bounded, uniformly continuous }h\in BUC\left(
\mathbb{R}
^{m}\right)  ,\text{ }\\
\text{ there exists }\eta_{h}:%
\mathbb{R}
\longrightarrow%
\mathbb{R}
_{+}\text{ s.t. }\eta_{h}\left(  0\right)  =0\text{ and }\\
\left\vert \int_{E}h\left(  \theta\right)  Q\left(  \gamma,d\theta\right)
-\int_{E}h\left(  \theta\right)  Q\left(  \gamma^{\prime},d\theta\right)
\right\vert \leq\eta_{h}\left(  \left\vert \gamma-\gamma^{\prime}\right\vert
\right)  .
\end{array}
\right.
\]
(The distance $\left\vert \gamma-\gamma^{\prime}\right\vert $ is the usual
Euclidian one on $%
\mathbb{R}
^{m}$.) Given an initial mode $\gamma_{0}\in E,$ the first jump time satisfies
$\mathbb{P}^{0,\gamma_{0}}\left(  T_{1}\geq t\right)  =\exp\left(
-t\lambda\left(  \gamma_{0}\right)  \right)  .$ The process $\gamma
_{t}:=\gamma_{0},$ on $t<T$ $_{1}.$ The post-jump location $\gamma^{1}$ has
$Q\left(  \gamma_{0},\cdot\right)  $ as conditional distribution. Next, we
select the inter-jump time $T_{2}-T_{1}$ such that $\mathbb{P}^{0,\gamma_{0}%
}\left(  T_{2}-T_{1}\geq t\text{ }/\text{ }T_{1},\gamma^{1}\right)
=\exp\left(  -t\lambda\left(  \gamma^{1}\right)  \right)  $ and set
$\gamma_{t}:=\gamma^{1},$ if $t\in\left[  T_{1},T_{2}\right)  .$ The post-jump
location $\gamma^{2}$ satisfies $\mathbb{P}^{0,\gamma_{0}}\left(  \gamma
^{2}\in A\text{ }/\text{ }T_{2},T_{1},\gamma^{1}\right)  =Q\left(  \gamma
^{1},A\right)  ,$ for all Borel set $A\subset E.$ And so on. Similar
construction can be given for a non-zero initial starting time (i.e a pair
$\left(  t,\gamma_{0}\right)  $).

We look at the process $\gamma$ under $\mathbb{P}^{0,\gamma_{0}}$ and denote
by $\mathbb{F}^{0}$ the filtration $\left(  \mathcal{F}_{\left[  0,t\right]
}:=\sigma\left\{  \gamma_{r}:r\in\left[  0,t\right]  \right\}  \right)
_{t\geq0}.$ The predictable $\sigma$-algebra will be denoted by $\mathcal{P}%
^{0}$ and the progressive $\sigma$-algebra by $Prog^{0}.$ As usual, we
introduce the random measure $p$ on $\Omega\times\left(  0,\infty\right)
\times E$ by setting%
\[
p\left(  \omega,A\right)  =\sum_{k\geq1}1_{\left(  T_{k}\left(  \omega\right)
,\gamma_{T_{k}\left(  \omega\right)  }\left(  \omega\right)  \right)  \in
A},\text{ for all }\omega\in\Omega,\text{ }A\in\mathcal{B}\left(
0,\infty\right)  \times\mathcal{B}\left(  E\right)  .
\]
The compensator of $p$ is $\lambda\left(  \gamma_{s-}\right)  Q\left(
\gamma_{s-},d\theta\right)  ds$ and the compensated martingale measure is
given by
\[
q\left(  ds,d\theta\right)  :=p\left(  ds,d\theta\right)  -\lambda\left(
\gamma_{s-}\right)  Q\left(  \gamma_{s-},d\theta\right)  ds.
\]
Following the general theory of integration with respect to random measures
(see, for example \cite{Ikeda_Watanabe_1981}), we denote by $\mathcal{L}%
^{r}\left(  p;%
\mathbb{R}
^{M}\right)  $ the space of all $\mathcal{P}^{0}\otimes\mathcal{B}\left(
E\right)  $ - measurable, $%
\mathbb{R}
^{M}-$valued functions $H_{s}\left(  \omega,\theta\right)  $ on $\Omega\times%
\mathbb{R}
_{+}\times E$ such that
\begin{align*}
&  \mathbb{E}^{0,\gamma_{0}}\left[  \int_{0}^{T}\int_{E}\left\vert
H_{s}\left(  \theta\right)  \right\vert ^{r}p\left(  ds,d\theta\right)
\right] \\
&  =\mathbb{E}^{0,\gamma_{0}}\left[  \int_{0}^{T}\int_{E}\left\vert
H_{s}\left(  \theta\right)  \right\vert ^{r}\lambda\left(  \gamma_{s-}\right)
Q\left(  \gamma_{s-},d\theta\right)  ds\right]  <\infty,
\end{align*}
for all $T<\infty.$Here, $M\in%
\mathbb{N}
^{\ast}$ and $r\geq1$ is \ a real parameter. By abuse of notation, whenever no
confusion is at risk, the family of processes satisfying the above condition
for a fixed $T>0$ will still be denoted by $\mathcal{L}^{r}\left(  p;%
\mathbb{R}
^{M}\right)  $.

\subsection{Switch Linear Model}

We consider switch systems given by a process $(X(t),\gamma(t))$ on the state
space $%
\mathbb{R}
^{n}\times E,$ for some $n\geq1$ and a family of modes $E$. $\ $The control
state space is assumed to be some Euclidian space $%
\mathbb{R}
^{d},$ $d\geq1$. The component $X(t)$ follows a differential dynamic depending
on the hidden variable $\gamma$. We will deal with the following model.%
\begin{equation}
\left\{
\begin{array}
[c]{l}%
dX_{s}^{x,u}=\left[  A\left(  \gamma_{s}\right)  X_{s}^{x,u}+B_{s}%
u_{s}\right]  ds+\int_{E}C\left(  \gamma_{s-},\theta\right)  X_{s-}%
^{x,u}q\left(  ds,d\theta\right)  ,\text{ }s\geq0,\\
X_{0}^{x,u}=x.
\end{array}
\right.  \label{SDE0}%
\end{equation}
The operators $A\left(  \gamma\right)  \in%
\mathbb{R}
^{n\times n}$ are such that $\sup_{\gamma\in E}\left\vert A\left(
\gamma\right)  \right\vert \leq a_{0}$, for some positive real constant
$a_{0}\in%
\mathbb{R}
$. The coefficient
\[
B\in\mathbb{L}_{loc}^{\infty}\left(  \Omega\times%
\mathbb{R}
_{+},\mathcal{P}^{0},\mathbb{P}^{0,\gamma_{0}}\otimes Leb;%
\mathbb{R}
^{n\times d}\right)  ,
\]
that is $B_{t}\left(  \omega\right)  \in%
\mathbb{R}
^{n\times d},$ for all $t\geq0$ and all $\omega\in\Omega$, the function $B$ is
$\mathcal{P}^{0}-$ measurable and locally bounded in time (i.e. $\sup
_{t\in\left[  0,T\right]  }\left\vert B_{t}\right\vert _{%
\mathbb{R}
^{n\times d}}<c_{T},$ $\mathbb{P}^{0,\gamma_{0}}-a.s.,$ for some positive real
constants $c_{T}$ and all $T>0$)$.$ Particular evolution will be imposed on
$B$ afterwards. The coefficients $C\left(  \gamma,\theta\right)  \in%
\mathbb{R}
^{n\times n}$, for all $\gamma,\theta\in E$ satisfy
\[
\sup_{\gamma\in E}\int_{E}\left\vert C\left(  \gamma,\theta\right)
\right\vert ^{2}\lambda\left(  \gamma\right)  Q\left(  \gamma,d\theta\right)
<\infty,\text{ for all }T>0.
\]
Moreover, the control process $u:\Omega\times%
\mathbb{R}
_{+}\longrightarrow%
\mathbb{R}
^{d}$ is an $%
\mathbb{R}
^{d}$-valued, $\mathbb{F}^{0}-$ progressively measurable, locally square
integrable process i.e.%
\[
\mathbb{E}^{0,\gamma_{0}}\left[  \int_{0}^{T}\left\vert u_{s}\right\vert
^{2}ds\right]  <\infty.
\]
The space of all such processes will be denoted by $\mathcal{U}_{ad}$ and
referred to as the family of admissible control processes. The explicit
structure of such processes can be found in \cite[Proposition 4.2.1]%
{Jacobsen}, for instance. Since the control process does not (directly)
intervene in the noise term, the solution of the above system can be
explicitly computed with $\mathcal{U}_{ad}$ processes instead of the (more
usual) predictable processes.

The reader acquainted with the literature on PDMP may note that our processes
are very particular. The evolutions concern full spaces and this is why jumps
from the border cannot occur. Between jumps, the couple $\left(
\gamma,x\right)  $ has $\left(  0,\left[  A\left(  \gamma\right)  -\int
C\left(  \gamma,\theta\right)  \lambda\left(  \gamma\right)  Q\left(
\gamma,d\theta\right)  \right]  x+Bu\right)  $ as speed, the jump intensity is
$\lambda\left(  \gamma\right)  $ and the post-jump measure is $\delta
_{x+C\left(  \gamma,\theta\right)  x}\left(  dy\right)  Q\left(
\gamma,d\theta\right)  $.

The reader is kindly asked to recall the financial example we gave in our
introduction. Within this framework, in order for our model to be coherent in
some/any finite time horizon $T>0$, it should be able to provide a "guess"
that is not very far from the actual price at time $T.$ In other words, by
eventually changing the control parameter $u$, we should be able to simulate
any possible (regular) price at time $T$ or, equivalently, any
square-integrable random variable which is measurable w.r.t. $\mathcal{F}%
_{\left[  0,T\right]  }$ (the natural $\sigma$-field generated by the random
measure $p$ prior to time $T$).

\begin{definition}
The system (\ref{SDE0}) is said to be approximately controllable in time $T>0
$ if, for every $\gamma_{0}\in E,$ every $\mathcal{F}_{\left[  0,T\right]  }%
$-measurable, square integrable $\xi\in\mathbb{L}^{2}\left(  \Omega
,\mathcal{F}_{\left[  0,T\right]  },\mathbb{P}^{0,\gamma_{0}};%
\mathbb{R}
^{n}\right)  $, every initial condition $x\in%
\mathbb{R}
^{n}$ and every $\varepsilon>0$, there exists some admissible control process
$u\in\mathcal{U}_{ad}$ such that $\mathbb{E}^{0,\gamma_{0}}\left[  \left\vert
X_{T}^{x,u}-\xi\right\vert ^{2}\right]  \leq\varepsilon.$

The system (\ref{SDE0}) is said to be approximately null-controllable in time
$T>0$ if the previous condition holds for $\xi=0$ ($\mathbb{P}^{0,\gamma_{0}}$-a.s.).
\end{definition}

For deterministic systems (cf. \cite{Hautus}, \cite{russell_Weiss_1994}, etc.)
or Brownian-driven control systems (e.g. \cite{Peng_94},
\cite{Buckdahn_Quincampoix_Tessitore_2006}, \cite{Liu_Peng_2002}),
controllability is known to be associated to observability properties for the
dual (linear) system. In the stochastic framework, one is led to backward
stochastic differential equations. Following this intuition, we consider the
following system.%
\begin{equation}
\left\{
\begin{array}
[c]{l}%
dY_{t}^{T,\xi}=\left[  -A^{\ast}\left(  \gamma_{t}\right)  Y_{t}^{T,\xi}%
-\int_{E}C^{\ast}\left(  \gamma_{t},\theta\right)  Z_{t}^{T,\xi}\left(
\theta\right)  \lambda\left(  \gamma_{t}\right)  Q\left(  \gamma_{t}%
,d\theta\right)  \right]  dt\\
\text{ \ \ \ \ \ \ \ }+\int_{E}Z_{t}^{T,\xi}\left(  \theta\right)  q\left(
dt,d\theta\right)  ,\\
Y_{T}^{T,\xi}=\xi\in\mathbb{L}^{2}\left(  \Omega,\mathcal{F}_{\left[
0,T\right]  },\mathbb{P}^{0,\gamma_{0}};%
\mathbb{R}
^{n}\right)  .
\end{array}
\right.  \label{BSDE0}%
\end{equation}
The solution is a couple
\[
\left(  Y_{\cdot}^{T,\xi},Z_{\cdot}^{T,\xi}\left(  \cdot\right)  \right)
\in\mathbb{L}^{2}\left(  \Omega\times\left[  0,T\right]  ,\text{ }%
Prog^{0},\mathbb{P}^{0,\gamma_{0}}\otimes Leb;%
\mathbb{R}
^{n}\right)  \times\mathcal{L}^{2}\left(  p;%
\mathbb{R}
^{n}\right)
\]
i.e. such that $Y_{\cdot}^{T,\xi}$ is progressively measurable and $Z_{\cdot
}^{T,\xi}$ is $\mathcal{P}^{0}\times\mathcal{B}\left(  E\right)  $-measurable
and
\[
\mathbb{E}^{0,\gamma_{0}}\left[  \int_{0}^{T}\left\vert Y_{s}^{T,\xi
}\right\vert ^{2}ds\right]  +\mathbb{E}^{0,\gamma_{0}}\left[  \int_{0}^{T}%
\int_{E}\left\vert Z_{s}^{T,\xi}\left(  \theta\right)  \right\vert ^{2}%
\lambda\left(  \gamma_{s-}\right)  Q\left(  \gamma_{s-},d\theta\right)
ds\right]  <\infty.
\]
For further details on existence and uniqueness of the solution of BSDE of
this type, the reader is referred to \cite{Confortola_Fuhrman_2014},
\cite{Confortola_Fuhrman_2013} (see also \cite{Xia_00}). (Although the results
in \cite{Confortola_Fuhrman_2014}, are given in the one-dimensional framework
for the purpose of the cited paper, the assertions in \cite[Lemmas 3.2 and
3.3, Theorems 2.2 and 3.4]{Confortola_Fuhrman_2014} hold true with no change
for the multi-dimensional case. Also the assumptions in \cite[Hypothesis
3.1]{Confortola_Fuhrman_2014} are easily verified in our model.)

We get the following duality result which specifies the connection between the
above equation (\ref{BSDE0}) and the controllability of equation (\ref{SDE0}).

\begin{theorem}
\label{TheoremCar}The system (\ref{SDE0}) is approximately controllable in
time $T>0$ if and only if, for every $\gamma_{0}\in E,$ the only solution
$\left(  Y_{t}^{T,\xi},Z_{t}^{T,\xi}\left(  \cdot\right)  \right)  $ of the
dual system (\ref{BSDE0}) for which $Y_{t}^{T,\xi}\in KerB_{t}^{\ast},$
$\mathbb{P}^{0,\gamma_{0}}\mathbb{\otimes}Leb$ almost everywhere on
$\Omega\times\left[  0,T\right]  $ is the trivial (zero) solution.

The necessary and sufficient condition for approximate null-controllability of
(\ref{SDE0}) is that for every $\gamma_{0}\in E,$ any solution $\left(
Y_{t}^{T,\xi},Z_{t}^{T,\xi}\left(  \cdot\right)  \right)  $ of the dual system
(\ref{BSDE0}) for which $Y_{t}^{T,\xi}\in KerB_{t}^{\ast}$ $,$ $\mathbb{P}%
^{0,\gamma_{0}}\mathbb{\otimes}Leb$ almost everywhere on $\Omega\times\left[
0,T\right]  $ should equally satisfy $Y_{0}^{T,\xi}=0.$
\end{theorem}

\begin{proof}
Let us fix $T>0$ and $\gamma_{0}\in E$. As usual, one can consider the
(continuous) linear operators
\[
\left.
\begin{array}
[c]{l}%
\mathcal{R}_{T}:\mathcal{U}_{ad}\longrightarrow\mathbb{L}^{2}\left(
\Omega,\mathcal{F}_{\left[  0,T\right]  },\mathbb{P}^{0,\gamma_{0}};%
\mathbb{R}
^{n}\right)  \text{, }\mathcal{R}_{T}\left(  u\right)  :=X_{T}^{0,u},\text{for
all }u\in\mathcal{U}_{ad},\\
\mathcal{I}_{T}:%
\mathbb{R}
^{n}\longrightarrow\mathbb{L}^{2}\left(  \Omega,\mathcal{F}_{\left[
0,T\right]  },\mathbb{P}^{0,\gamma_{0}};%
\mathbb{R}
^{n}\right)  ,\text{ }\mathcal{I}_{T}\left(  x\right)  =X_{T}^{x,0},\text{ for
all }x\mathcal{\in}%
\mathbb{R}
^{n}.
\end{array}
\right.
\]
One notices easily that $X_{T}^{x,u}=\mathcal{R}_{T}\left(  u\right)
+\mathcal{I}_{T}\left(  x\right)  ,$ for all $u\in\mathcal{U}_{ad}$ and all
$x\mathcal{\in}%
\mathbb{R}
^{n}.$ Hence, approximate controllability is equivalent to $\mathcal{R}%
_{T}(\mathcal{U}_{ad})$ being dense in $\mathbb{L}^{2}\left(  \Omega
,\mathcal{F}_{\left[  0,T\right]  },\mathbb{P}^{0,\gamma_{0}};%
\mathbb{R}
^{n}\right)  ,$ while approximate null-controllability is equivalent to
$\mathcal{I}_{T}\left(
\mathbb{R}
^{n}\right)  $ being included in the closure of $\mathcal{R}_{T}%
(\mathcal{U}_{ad})$ (w.r.t. the usual topology on $\mathbb{L}^{2}\left(
\Omega,\mathcal{F}_{\left[  0,T\right]  },\mathbb{P}^{0,\gamma_{0}};%
\mathbb{R}
^{n}\right)  $). Equivalently, this leads to
\[
Ker\left(  \mathcal{R}_{T}^{\ast}\right)  =\left\{  0\right\}  \text{(resp.
}Ker\left(  \mathcal{R}_{T}^{\ast}\right)  \subset Ker\left(  \mathcal{I}%
_{T}^{\ast}\right)  \text{).}%
\]
One easily gets, using It\^{o}'s formula (the reader may consult \cite[Chapter
II, Section 5, Theorem 5.1]{Ikeda_Watanabe_1981}, for example),%
\[
\mathbb{E}^{0,\gamma_{0}}\left[  \left\langle X_{T}^{x,u},Y_{T}^{T,\xi
}\right\rangle \right]  =\left\langle x,Y_{0}^{T,\xi}\right\rangle
+\mathbb{E}^{0,\gamma_{0}}\left[  \int_{0}^{T}\left\langle B_{t}u_{t}%
,Y_{t}^{T,\xi}\right\rangle dt\right]  ,
\]
for all $T>0,$ $x\in%
\mathbb{R}
^{n},$ $u\in\mathcal{U}_{ad}$ and all $\xi\in\mathbb{L}^{2}\left(
\Omega,\mathcal{F}_{\left[  0,T\right]  },\mathbb{P}^{0,\gamma_{0}};%
\mathbb{R}
^{n}\right)  .$ For $x=0,$ one identifies
\[
\mathcal{R}_{T}^{\ast}\left(  \xi\right)  =\left(  B_{t}^{\ast}Y_{t}^{T,\xi
}\right)  _{0\leq t\leq T}.
\]
For $u=0,$ one gets%
\[
\mathcal{I}_{T}^{\ast}\left(  \xi\right)  =Y_{0}^{T,\xi}.
\]
The assertion of our theorem follows.
\end{proof}

Equivalent assertions are easily obtained by interpreting the system
(\ref{BSDE0}) as a controlled, forward one :%
\begin{equation}
\left\{
\begin{array}
[c]{l}%
dY_{t}^{y,v}=\left(  -A^{\ast}\left(  \gamma_{t}\right)  Y_{t}^{y,v}-\int%
_{E}C^{\ast}\left(  \gamma_{t},\theta\right)  v_{t}\left(  \theta\right)
\lambda\left(  \gamma_{t}\right)  Q\left(  \gamma_{t},d\theta\right)  \right)
dt\\
\text{ \ \ \ \ \ \ \ }+\int_{E}v_{t}\left(  \theta\right)  q\left(
dt,d\theta\right)  ,\\
Y_{0}^{y,v}=y\in%
\mathbb{R}
^{n}.
\end{array}
\right.  \label{SDE'}%
\end{equation}
We emphasize that in this framework, the family of admissible control
processes is given by $v\in\mathcal{L}^{2}\left(  p;%
\mathbb{R}
^{n}\right)  $. This remark will be of particular use for sufficiency criteria.

\begin{theorem}
\label{ThEquivCtrl}The system (\ref{SDE0}) is approximately controllable in
time $T>0$ if and only if, for every initial data $\gamma_{0}\in E,$ $y\in%
\mathbb{R}
^{n}$ and every predictable control process $v\in\mathcal{L}^{2}\left(  p;%
\mathbb{R}
^{n}\right)  $ such that $Y_{t}^{y,v}\in KerB_{t}^{\ast},$ $\mathbb{P}%
^{0,\gamma_{0}}\mathbb{\otimes}Leb$ almost everywhere on $\Omega\times\left[
0,T\right]  ,$ it must hold $Y_{t}^{y,v}=0,$ $\mathbb{P}^{0,\gamma_{0}%
}\mathbb{\otimes}Leb$-almost everywhere on $\Omega\times\left[  0,T\right]  .$

The system (\ref{SDE0}) is approximately null-controllable in time $T>0$ if
and only if, for every initial data $\gamma_{0}\in E,$ $y\in%
\mathbb{R}
^{n}$ and every predictable control process $v\in\mathcal{L}^{2}\left(  p;%
\mathbb{R}
^{n}\right)  $ such that $Y_{t}^{y,v}\in KerB_{t}^{\ast},$ $\mathbb{P}%
^{0,\gamma_{0}}\mathbb{\otimes}Leb$ almost everywhere on $\Omega\times\left[
0,T\right]  ,$ it must hold $y=0$.
\end{theorem}

\begin{remark}
The reader should note that for every initial data $y\in%
\mathbb{R}
^{n}$ and $v\in\mathcal{L}^{2}\left(  p;%
\mathbb{R}
^{n}\right)  $, one identifies the couple $\left(  Y_{\cdot}^{y,v},v_{\cdot
}\left(  \cdot\right)  \right)  $ with the solution $\left(  Y_{\cdot}^{T,\xi
},Z_{\cdot}^{T,\xi}\left(  \cdot\right)  \right)  $ for $\xi=Y_{T}^{y,v}.$
Conversly, for every $\xi\in\mathbb{L}^{2}\left(  \Omega,\mathcal{F}_{\left[
0,T\right]  },\mathbb{P}^{0,\gamma_{0}};%
\mathbb{R}
^{n}\right)  ,$ the solution of (\ref{BSDE0}) is given by the couple $\left(
Y_{\cdot}^{y,v},v_{\cdot}\left(  \cdot\right)  \right)  $ by picking
$y=Y_{0}^{T,\xi}$ and $v=Z^{T,\xi}.$ Thus, the previous Theorem is only a
reformulation of Theorem \ref{TheoremCar}.
\end{remark}

\section{Explicit Invariance-Type Conditions\label{SectionExplicitInvariance}}

As one can easily see, the approximate (null-) controllability conditions
emphasized in the previous result involve $Ker\left(  B_{t}^{\ast}\right)  .$
In order to obtain the (orthogonal) projection onto this space, one uses the
Moore-Penrose pseudoinverse of $B_{t}$%
\[
B_{t}^{+}:=\lim_{\delta\rightarrow0+}\left(  B_{t}^{\ast}B_{t}+\delta
I\right)  ^{-1}B_{t}^{\ast}%
\]
(given here by means of Tykhonov regularizations). However, even if $B_{t}$ is
continuous, $B_{t}^{+}$ can only have continuity properties if these matrices
have the same rank. In this framework, the results on derivation of $B^{+}%
$given in \cite{Golub_Pereyra} apply. To simplify the presentation (by
avoiding fastidious rank conditions), we assume, throughout the remaining of
the paper (and unless stated otherwise), $B$ to satisfy a linear system%
\[
\left\{
\begin{array}
[c]{c}%
dB_{t}=\beta\gamma_{t}B_{t}dt\\
B_{0}=B^{0}\left(  \gamma_{0}\right)  .
\end{array}
\right.  ,
\]
where $\beta\in%
\mathbb{R}
^{1\times d}$ and the space $E$ to be compact$.$ Then, one computes explicitly
the predictable process $B$
\[
B_{t}=e^{\int_{0}^{t}\beta\gamma_{s}ds}B^{0}\left(  \gamma_{0}\right)  ,\text{
}B_{t}^{+}=e^{-\int_{0}^{t}\beta\gamma_{s}ds}\left(  B^{0}\left(  \gamma
_{0}\right)  \right)  ^{+},\text{ for all }t\geq0.
\]
It is clear that $Rank\left(  B_{t}\right)  =Rank\left(  B^{0}\left(
\gamma_{0}\right)  \right)  ,$ for all $t\geq0$. Moreover, $B_{t}B_{t}^{+}%
=\Pi_{\left(  Ker\left(  B_{t}^{\ast}\right)  \right)  ^{\bot}}$ is the
orthogonal projector onto the orthogonal complement of the kernel of
$B_{t}^{\ast}$. In particular,
\[
\Pi_{Ker\left(  B_{t}^{\ast}\right)  }=I-B^{0}\left(  \gamma_{0}\right)
\left(  B^{0}\left(  \gamma_{0}\right)  \right)  ^{+}=\Pi_{Ker\left[  \left(
B^{0}\left(  \gamma_{0}\right)  \right)  ^{\ast}\right]  }.
\]

We recall the following notions of invariance (see \cite{Curtain_86},
\cite{Schmidt_Stern_80}).

\begin{definition}
We consider two linear operators $\mathcal{A\in}%
\mathbb{R}
^{n\times n}$ and $\mathcal{C}$ defined on some vector space $\mathcal{X}$ and
taking its values in $%
\mathbb{R}
^{n}$.

(a) A set $V\subset%
\mathbb{R}
^{n}$ is said to be $\mathcal{A}$- invariant if $\mathcal{A}V\subset V.$

(b) A set $V\subset%
\mathbb{R}
^{n}$ is said to be $\left(  \mathcal{A};\mathcal{C}\right)  $- invariant if
$\mathcal{A}V\subset V+\operatorname{Im}\mathcal{C}.$

(c) A set $V\subset%
\mathbb{R}
^{n}$ is said to be feedback $\left(  \mathcal{A};\mathcal{C}\right)  $-
invariant if there exists some linear operator $\mathcal{F}:%
\mathbb{R}
^{n}\longrightarrow\mathcal{X}$ such that $\left(  \mathcal{A}+\mathcal{CF}%
\right)  V\subset V$ (i.e. $V$ is $\mathcal{A}+\mathcal{CF}$- invariant).
\end{definition}

If the space $\mathcal{X}$ is $%
\mathbb{R}
^{n}$ or some functional space with values in $%
\mathbb{R}
^{n}$ (of $\mathbb{L}^{p}$-type, for example), one can also define a notion of
$\left(  \mathcal{A};\mathcal{C}\right)  $- strict invariance by imposing that
$\mathcal{A}V\subset V+\operatorname{Im}\left(  \mathcal{C}\Pi_{V}\right)  .$
Of course, any strictly invariant space is also invariant.

We will now concentrate on some explicit algebraic conditions which are either
necessary or sufficient for the approximate (null-)controllability. These
criteria apply to general systems (without specific conditions on the
coefficients). In the next section, we will discuss particular cases when
these conditions are necessary and sufficient. We equally present several
examples showing to which extent our conditions are necessary and sufficient.

\subsection{Necessary Invariance Conditions\label{SubsectionNecessary}}

\subsubsection{First Necessary Condition}

We begin with some necessary invariance conditions in order to have
controllability for the system (\ref{SDE0}). Without any further assumptions
on $E$ and $Q,$ we have the following.

\begin{proposition}
\label{PropNec1}If the system (\ref{SDE0}) is approximately null-controllable,
then, for every $\gamma_{0}\in E,$ the largest subspace of $Ker\left[  \left(
B^{0}\left(  \gamma_{0}\right)  \right)  ^{\ast}\right]  $ which is $A^{\ast
}\left(  \gamma_{0}\right)  -\lambda\left(  \gamma_{0}\right)  \int_{E}%
C^{\ast}\left(  \gamma_{0},\theta\right)  Q\left(  \gamma_{0},d\theta\right)
$ - invariant is reduced to the trivial subspace $\left\{  0\right\}  .$
\end{proposition}

\begin{proof}
Let us assume that our system is approximately null-controllable at time $T>0
$ and let us fix $x\in%
\mathbb{R}
^{n}$ and $\gamma_{0}\in E.$ Since $\gamma_{0}$ is fixed, we drop the
dependence of $\gamma_{0}$ in $\mathbb{P}^{0,\gamma_{0}}$. Then, for every
$\varepsilon>0,$ there exists some admissible control process $u\in
\mathcal{U}_{ad}$ such that $\mathbb{E}\left[  \left\vert X_{T}^{x,u}%
\right\vert ^{2}\right]  \leq\varepsilon.$ Since $u$ is assumed to be
progressively measurable, it can be identified prior to jump time $T_{1},$
with some deterministic borelian control still denoted by $u$ (see, for
example, \cite[Proposition 4.2.1]{Jacobsen})$.$ One notes easily that
\begin{align*}
\varepsilon &  \geq\mathbb{E}\left[  \left\vert X_{T}^{x,u}\right\vert
^{2}\right]  \geq\mathbb{E}\left[  \left\vert X_{T}^{x,u}\right\vert
^{2}1_{T_{1}>T}\right]  =\left\vert \Phi_{\gamma_{0}}\left(  T;x,u\right)
\right\vert ^{2}\mathbb{P}\left(  T_{1}>T\right) \\
&  \geq\left\vert \Phi_{\gamma_{0}}\left(  T;x,u\right)  \right\vert
^{2}e^{-Tc_{0}},
\end{align*}
where $c_{0}$ is an upper bound for the transition intensities $\lambda$ and
$\Phi_{\gamma_{0}}$ satisfies the deterministic equation%
\[
\left\{
\begin{array}
[c]{l}%
d\Phi_{\gamma_{0}}\left(  s;x,u\right)  =\left[
\begin{array}
[c]{c}%
\left(  A\left(  \gamma_{0}\right)  -\lambda\left(  \gamma_{0}\right)
\int_{E}C\left(  \gamma_{0},\theta\right)  Q\left(  \gamma_{0},d\theta\right)
\right)  \Phi_{\gamma_{0}}\left(  s;x,u\right) \\
+B^{0}\left(  \gamma_{0}\right)  \left(  e^{\beta\gamma_{0}s}u_{s}\right)
\end{array}
\right]  ds,\\
\Phi_{\gamma_{0}}\left(  0;x,u\right)  =x.
\end{array}
\right.
\]
Then this linear system with mode $\gamma_{0}$ is controllable (keep in mind
that since the coefficients are time-invariant and we have finite-dimensional
state space, the different notions of controllability coincide). Kalman's
criterion of controllability yields%
\[
Rank\left[  \mathcal{B}\text{, }\mathcal{AB},\text{ }...\text{, }%
\mathcal{A}^{n-1}\mathcal{B}\right]  =n,
\]
where $\mathcal{A=}A\left(  \gamma_{0}\right)  -\lambda\left(  \gamma
_{0}\right)  \int_{E}C\left(  \gamma_{0},\theta\right)  Q\left(  \gamma
_{0},d\theta\right)  $ and $\mathcal{B=}B^{0}\left(  \gamma_{0}\right)  $ and
this implies our conclusion.
\end{proof}

In the next section, we shall exhibit various frameworks (continuous switching
systems, corresponding to the case $C=0$) in which this condition is equally
sufficient for the approximate null-controllability as well as counterexamples
for the general setting.

\subsubsection{Second Necessary Condition}

Whenever $\gamma\in E,$ the operator $\mathcal{C}^{\ast}\left(  \gamma\right)
\in\mathcal{L}\left(  \mathbb{L}^{2}\left(  E,\mathcal{B}\left(  E\right)
,\lambda\left(  \gamma\right)  Q\left(  \gamma,\cdot\right)  ;%
\mathbb{R}
^{n}\right)  ;%
\mathbb{R}
^{n}\right)  $ (the space of linear operators from $\mathbb{L}^{2}\left(
E,\mathcal{B}\left(  E\right)  ,\lambda\left(  \gamma\right)  Q\left(
\gamma,\cdot\right)  ;%
\mathbb{R}
^{n}\right)  $ to $%
\mathbb{R}
^{n})$ is defined by%
\[
\mathcal{C}^{\ast}\left(  \gamma\right)  \phi=\int_{E}C^{\ast}\left(
\gamma,\theta\right)  \phi\left(  \theta\right)  \lambda\left(  \gamma\right)
Q\left(  \gamma,d\theta\right)  ,
\]
for all $\phi\in\mathbb{L}^{2}\left(  E,\mathcal{B}\left(  E\right)
,\lambda\left(  \gamma\right)  Q\left(  \gamma,\cdot\right)  ;%
\mathbb{R}
^{n}\right)  .$

Let us now assume the set of modes $E$ to be finite. This assumption is made
in order to simplify the arguments on the local square integrability of
feedback controls in the proof of Proposition \ref{PropNec2}. We define
\[
Acc\left(  k,\gamma_{0}\right)  =\left\{  \gamma\in E:\exists0\leq k^{\prime
}\leq k\text{ s.t. }Q^{k^{\prime}}\left(  \gamma_{0};\gamma^{\prime}\right)
>0\right\}  ,
\]
the family of accessible modes in at most $k\geq0$ iterations starting from
$\gamma_{0}.$ (In the infinite setting, accessible states can alternatively be
defined using the transition kernel). We use the obvious convention
$Acc\left(  0,\gamma_{0}\right)  =\left\{  \gamma_{0}\right\}  $. We set
$V_{k}$ to be the largest subspace $V\subset Ker\left[  \left(  B^{0}\left(
\gamma_{0}\right)  \right)  ^{\ast}\right]  $ such that
\[
A^{\ast}\left(  \gamma\right)  V\subset V+\operatorname{Im}\left(
\mathcal{C}^{\ast}\left(  \gamma\right)  \Pi_{V}\right)  =V+\mathcal{C}^{\ast
}\left(  \gamma\right)  \left(  \mathbb{L}^{2}\left(  E,\mathcal{B}\left(
E\right)  ,\lambda\left(  \gamma\right)  Q\left(  \gamma,\cdot\right)
;V\right)  \right)  ,
\]
for all $\gamma\in Acc\left(  k,\gamma_{0}\right)  $ (i.e. the largest space
$\left(  A^{\ast}\left(  \gamma\right)  ;\mathcal{C}^{\ast}\left(
\gamma\right)  \right)  $-strictly invariant for all $\gamma\in Acc\left(
k,\gamma_{0}\right)  $). All these subspaces depend on the initial mode
$\gamma_{0}.$ Whenever no confusion is at risk, this dependency is dropped.

\begin{remark}
(i) The reader is invited to note that $V_{0}$ is the largest subspace
$V\subset Ker\left[  \left(  B^{0}\left(  \gamma_{0}\right)  \right)  ^{\ast
}\right]  $ which is $\left(  A^{\ast}\left(  \gamma_{0}\right)
;\mathcal{C}^{\ast}\left(  \gamma_{0}\right)  \Pi_{V}\right)  $-invariant.

(ii) If the family of modes $E$ is finite, then $V_{k}=V_{card\left(
E\right)  },$ for all $k\geq card\left(  E\right)  .$
\end{remark}

The following condition is equally necessary in order to have controllability.

\begin{proposition}
\label{PropNec2}If the system (\ref{SDE0}) \ is approximately
null-controllable, then the subspace $V_{\infty}:=\underset{k\geq0}{\cap}%
V_{k}$ is reduced to $\left\{  0\right\}  .$
\end{proposition}

\begin{proof}
It is obvious from the definition that the family $\left(  V_{k}\right)
_{k\geq0}$ is decreasing. We set $\mathcal{C}_{\infty}^{\ast}\left(
\gamma\right)  :=\mathcal{C}^{\ast}\left(  \gamma\right)  \Pi_{V_{\infty}}$,
for all $\gamma\in E.$ Then, due to the definition of $V_{\infty},$ one gets
that $V_{\infty}$ is $\left(  A^{\ast}\left(  \gamma\right)  ;\mathcal{C}%
_{\infty}^{\ast}\left(  \gamma\right)  \right)  $-invariant
\[
A^{\ast}\left(  \gamma\right)  V_{\infty}\subset V_{\infty}+\operatorname{Im}%
\mathcal{C}_{\infty}^{\ast}\left(  \gamma\right)  ,
\]
for every $\gamma\in\underset{k\geq0}{\cup}Acc\left(  k,\gamma_{0}\right)  .$
Using \cite[Theorem 3.2]{Schmidt_Stern_80} (see also \cite[Lemma
4.6]{Curtain_86}), the set $V_{\infty}$ is equally $\left(  A^{\ast}\left(
\gamma\right)  ;\mathcal{C}_{\infty}^{\ast}\left(  \gamma\right)  \right)
-$feedback invariant. Thus, there exists some (bounded) linear operator
\[
F\left(  \gamma\right)  :%
\mathbb{R}
^{n}\longrightarrow\mathbb{L}^{2}\left(  E,\mathcal{B}\left(  E\right)
,\lambda\left(  \gamma\right)  Q\left(  \gamma,\cdot\right)  ;%
\mathbb{R}
^{n}\right)
\]
such that $V_{\infty}$ is $A^{\ast}\left(  \gamma\right)  +\mathcal{C}^{\ast
}\left(  \gamma\right)  \Pi_{V_{\infty}}F\left(  \gamma\right)  $- invariant.
We consider the linear stochastic system%
\[
\left\{
\begin{array}
[c]{l}%
dY_{t}^{y,\gamma_{0}}=-\left(  A^{\ast}\left(  \gamma_{t}\right)  +\int%
_{E}C^{\ast}\left(  \gamma_{t},\theta\right)  \Pi_{V_{\infty}}F\left(
\gamma_{t}\right)  \left(  \theta\right)  \lambda\left(  \gamma_{t}\right)
Q\left(  \gamma_{t},d\theta\right)  \right)  Y_{t}^{y,\gamma_{0}}dt\\
\text{ \ \ \ \ \ \ \ \ \ \ }+\int_{E}\Pi_{V_{\infty}}F\left(  \theta\right)
Y_{t-}^{y,\gamma_{0}}q\left(  dt,d\theta\right)  ,\text{ }t\geq0,\\
Y_{0}^{y,\gamma_{0}}=y\in V_{\infty}\subset%
\mathbb{R}
^{n}.
\end{array}
\right.
\]
Then, it is clear that $Y_{t}^{y,\gamma_{0}}=Y_{t}^{y,v^{feedback}}$, where
$Y_{t}^{y,v^{feedback}}$ is the unique solution of (\ref{SDE'}) associated to
the feedback control
\[
v_{t}^{feedback}\left(  \theta\right)  :=\Pi_{V_{\infty}}F\left(
\theta\right)  Y_{t-}^{y,\gamma_{0}},\text{ }t\geq0.
\]
We wish to emphasize that the solution $Y_{\cdot}^{y,\gamma_{0}}$ will be
adapted and c\`{a}dl\`{a}g and, thus, $v_{\cdot}^{feedback}$ is an admissible
control (for details on the structure of predictable processes, we refer the
reader to \cite[Equation 26.4]{davis_93}, \cite[Proposition 4.2.1]{Jacobsen}
or \cite[Appendix A2, Theorem T34]{Bremaud_1981}). For every initial datum
$y\in V_{\infty},$ it is clear that the solution $Y_{\cdot}^{y,\gamma_{0}}\in
V_{\infty},$ $\mathbb{P}^{0,\gamma_{0}}\otimes Leb$ - almost everywhere. In
particular, $B_{t}^{\ast}Y_{t}^{y,v^{feedback}}=0,$ $\mathbb{P}^{0,\gamma_{0}%
}\otimes Leb$ - almost everywhere. If our system is approximately
null-controllable, due to Theorem \ref{ThEquivCtrl}, one has $y=0$ and the
conclusion of our proposition follows recalling that $y\in V_{\infty}$ is arbitrary.
\end{proof}

\subsection{Sufficient Invariance Conditions\label{SubsectionSufficient}}

As we did for $C^{\ast}$ by identifying $C^{\ast}\left(  \gamma,\cdot\right)
$ with a linear operator
\[
\mathcal{C}^{\ast}\left(  \gamma\right)  \in\mathcal{L}\left(  \mathbb{L}%
^{2}\left(  E,\mathcal{B}\left(  E\right)  ,\lambda\left(  \gamma\right)
Q\left(  \gamma,\cdot\right)  ;%
\mathbb{R}
^{n}\right)  ;%
\mathbb{R}
^{n}\right)  ,
\]
we let $\mathcal{I}\left(  \gamma\right)  \in\mathcal{L}\left(  \mathbb{L}%
^{2}\left(  E,\mathcal{B}\left(  E\right)  ,\lambda\left(  \gamma\right)
Q\left(  \gamma,\cdot\right)  ;%
\mathbb{R}
^{n}\right)  ;%
\mathbb{R}
^{n}\right)  $ be associated to the identity (i.e. we define it by setting%
\[
\mathcal{I}\left(  \gamma\right)  \phi=\int_{E}\phi\left(  \theta\right)
\lambda\left(  \gamma\right)  Q\left(  \gamma,d\theta\right)  ,
\]
for all $\phi\in\mathbb{L}^{2}\left(  E,\mathcal{B}\left(  E\right)
,\lambda\left(  \gamma\right)  Q\left(  \gamma,\cdot\right)  ;%
\mathbb{R}
^{n}\right)  $)$.$ As we have seen, the (second) necessary condition for
approximate null-controllability involves some subspace $V_{\infty}\subset
Ker\left(  \left[  B^{0}\left(  \gamma_{0}\right)  \right]  ^{\ast}\right)  $
such that
\[
A^{\ast}\left(  \gamma\right)  V_{\infty}\subset V_{\infty}+\operatorname{Im}%
\mathcal{C}_{\infty}^{\ast}\left(  \gamma\right)  =V_{\infty}%
+\operatorname{Im}\left(  \mathcal{C}^{\ast}\left(  \gamma\right)
\Pi_{V_{\infty}}\right)  ,
\]
at least for $\gamma=\gamma_{0}.$ A slightly stronger (yet similar) condition
written for $\gamma_{0}$ turns out to imply approximate null-controllability.

\begin{proposition}
\label{PropSuf1}If, for every $\gamma_{0}\in E,$ the largest subspace of
$\ Ker\left(  \left[  B^{0}\left(  \gamma_{0}\right)  \right]  ^{\ast}\right)
$ which is $\left(  A^{\ast}\left(  \gamma_{0}\right)  ;\left(  \mathcal{C}%
^{\ast}\left(  \gamma_{0}\right)  +\mathcal{I}\left(  \gamma_{0}\right)
\right)  \Pi_{Ker\left(  \left[  B^{0}\left(  \gamma_{0}\right)  \right]
^{\ast}\right)  }\right)  $ - invariant is reduced to $\left\{  0\right\}  $,
then the system (\ref{SDE0}) is approximately null-controllable.
\end{proposition}

\begin{proof}
For fixed $\gamma_{0}\in E$, we drop the dependency on $\gamma_{0}$ in
$\mathbb{P}^{0,\gamma_{0}}.$ For $y\in%
\mathbb{R}
^{n}$ and $v\in\mathcal{L}^{2}\left(  p;%
\mathbb{R}
^{n}\right)  ,$ we let $Y^{y,v}$ denote the unique controlled solution of
(\ref{SDE'}). It\^{o}'s formula (see \cite{Ikeda_Watanabe_1981} or
\cite[Theorem 31.3]{davis_93}) yields%
\begin{align*}
&  d\left[  \left(  I-\Pi_{Ker\left(  B_{t}^{\ast}\right)  }\right)
Y_{t}^{y,v}\right] \\
&  =\int_{E}\left(  I-\Pi_{Ker\left(  B_{t}^{\ast}\right)  }\right)
v_{t}\left(  \theta\right)  q\left(  dt,d\theta\right)  -\left(
I-\Pi_{Ker\left(  B_{t}^{\ast}\right)  }\right)  A^{\ast}\left(  \gamma
_{t}\right)  Y_{t}^{y,v}dt\\
&  -\int_{E}\left(  I-\Pi_{Ker\left(  B_{t}^{\ast}\right)  }\right)  C^{\ast
}\left(  \theta\right)  v_{t}\left(  \theta\right)  \lambda\left(  \gamma
_{t}\right)  Q\left(  \gamma_{t},d\theta\right)  dt.
\end{align*}
Whenever $B_{t}^{\ast}Y_{t}^{y,v}=0,$ $\mathbb{P\otimes}Leb$ almost everywhere
on $\Omega\times\left[  0,T\right]  ,$ one easily gets (by computing quadratic
variation)
\[
0=\mathbb{E}\left[  \int_{0}^{t}\int_{E}\left\vert \left(  I-\Pi_{Ker\left(
B_{s}^{\ast}\right)  }\right)  v_{s}\left(  \theta\right)  \right\vert
^{2}\lambda\left(  \gamma_{s}\right)  Q\left(  \gamma_{s},d\theta\right)
ds\right]  .
\]
We recall that $v$ is a predictable process and (see \cite[Equation
26.4]{davis_93}, \cite[Proposition 4.2.1]{Jacobsen} or \cite[Appendix A2,
Theorem T34]{Bremaud_1981} for further details)
\[
v_{t}\left(  \theta\right)  =v_{1}\left(  \theta,t\right)  1_{t\leq T_{1}%
}+\sum_{n\geq2}v_{n}\left(  \theta,t,\gamma^{1},...,\gamma^{n-1}\right)
1_{T_{n-1}<t\leq T_{n}},
\]
for some (deterministic) measurable functions $v_{n}.$ In particular, if
$\lambda\left(  \gamma_{0}\right)  >0$ (otherwise, we have a deterministic
system)$,$ then
\begin{equation}
v_{1}\left(  \theta,t\right)  \in Ker\left(  B_{t}^{\ast}\right)  =Ker\left(
\left[  B^{0}\left(  \gamma_{0}\right)  \right]  ^{\ast}\right)  , \label{eq1}%
\end{equation}
for $Q\left(  \gamma_{0},\cdot\right)  \otimes Leb$ almost all $\left(
\theta,t\right)  \in E\times%
\mathbb{R}
_{+}.$ \ We recall that, up to $T_{1}$, $Y^{y,v}$ coincides with $\phi^{y,v}$,
the solution of the deterministic system%
\[
\left\{
\begin{array}
[c]{l}%
d\phi_{t}^{y,v}=\left(  -A^{\ast}\left(  \gamma_{0}\right)  \phi_{t}%
^{y,v}-\int_{E}\left(  C^{\ast}\left(  \gamma_{0},\theta\right)  +I\right)
v_{1}\left(  \theta,t\right)  \lambda\left(  \gamma_{0}\right)  Q\left(
\gamma_{0},d\theta\right)  \right)  dt,\\
t\geq0,\\
\phi_{0}^{y,v}=y\in%
\mathbb{R}
^{n}.
\end{array}
\right.
\]
As a consequence of the fact that $Y^{y,v}\in Ker\left(  \left[  B^{0}\left(
\gamma_{0}\right)  \right]  ^{\ast}\right)  $ and $T_{1}$ is exponentially
distributed, it follows that
\begin{equation}
\left(  A^{\ast}\left(  \gamma_{0}\right)  \phi_{t}^{y,v}+\int_{E}\left(
C^{\ast}\left(  \theta\right)  +I\right)  v_{1}\left(  \theta,t\right)
\lambda\left(  \gamma_{0}\right)  Q\left(  \gamma_{0},d\theta\right)  \right)
\in Ker\left(  \left[  B^{0}\left(  \gamma_{0}\right)  \right]  ^{\ast
}\right)  , \label{eq2}%
\end{equation}
for almost all $t\geq0.$ At this point we recall the iterative construction of
invariant spaces. To this purpose, let us introduce the linear subspace
\begin{equation}
V^{0}=\left\{
\begin{array}
[c]{l}%
y^{\prime}\in Ker\left[  \left(  B^{0}\left(  \gamma_{0}\right)  \right)
^{\ast}\right]  :\\
\exists u\in\mathbb{L}^{2}\left(  E,\mathcal{B}\left(  E\right)  ,Q\left(
\gamma_{0},d\theta\right)  ;Ker\left[  \left(  B^{0}\left(  \gamma_{0}\right)
\right)  ^{\ast}\right]  \right)  s.t.\\
A^{\ast}\left(  \gamma_{0}\right)  y^{\prime}+\left(  \mathcal{C}^{\ast
}\left(  \gamma_{0}\right)  +\mathcal{I}\left(  \gamma_{0}\right)  \right)
u\in Ker\left[  \left(  B^{0}\left(  \gamma_{0}\right)  \right)  ^{\ast
}\right]  .
\end{array}
\right\}  \label{defV0}%
\end{equation}
Then, (\ref{eq1}) and (\ref{eq2}) imply that $\phi_{t}^{y,v}\in V^{0}$, for
almost all $t\geq0.$ Hence, using the linear character of $V^{0}$, we infer
\[
\left(  A^{\ast}\left(  \gamma_{0}\right)  \phi_{t}^{y,v}+\int_{E}\left(
C^{\ast}\left(  \gamma_{0},\theta\right)  +I\right)  v_{1}\left(
\theta,t\right)  \lambda\left(  \gamma_{0}\right)  Q\left(  \gamma_{0}%
,d\theta\right)  \right)  \in V^{0},
\]
for almost all $t\geq0.$ We then define%
\[
V^{1}=\left\{
\begin{array}
[c]{c}%
y^{\prime}\in V^{0}:\exists u\in\mathbb{L}^{2}\left(  E,\mathcal{B}\left(
E\right)  ,Q\left(  \gamma_{0},d\theta\right)  ;Ker\left[  \left(
B^{0}\left(  \gamma_{0}\right)  \right)  ^{\ast}\right]  \right)  s.t.\\
A^{\ast}\left(  \gamma_{0}\right)  y^{\prime}+\left(  \mathcal{C}^{\ast
}\left(  \gamma_{0}\right)  +\mathcal{I}\left(  \gamma_{0}\right)  \right)
u\in V^{0}.
\end{array}
\right\}
\]
and we deduce, as before, that $\phi_{t}^{y,v}\in V^{1}$, for almost all
$t\geq0,$ and so on for $V^{i},$ $i>1.$ We recall that $V^{i}\subset%
\mathbb{R}
^{n}.$ Hence, $V:=\underset{i=0,n}{\cap}V^{i}$ is a linear subspace of
$Ker\left[  \left(  B^{0}\left(  \gamma_{0}\right)  \right)  ^{\ast}\right]  $
which is $\left(  A^{\ast}\left(  \gamma_{0}\right)  ;\mathcal{C}^{\ast
}\left(  \gamma_{0}\right)  +\mathcal{I}\left(  \gamma_{0}\right)  \right)
$-invariant and such that $\phi_{t}^{y,v}\in V$, for almost all $t\geq0$. By
assumption, $V=\left\{  0\right\}  .$ Thus, $\phi_{t}^{y,v}=0$ for almost all
$t\geq0$ and, due to the continuity, $y=0.$ In particular, it follows that,
whenever $B_{t}^{\ast}Y_{t}^{y,v}=0,$ $\mathbb{P\otimes}Leb$ almost everywhere
on $\Omega\times\left[  0,T\right]  $, one has to have $y=0$ i.e. our system
is approximately null-controllable.
\end{proof}

\begin{remark}
\label{ViabKernel}(i) If $W$ is the largest subspace of $Ker\left[  \left(
B^{0}\left(  \gamma_{0}\right)  \right)  ^{\ast}\right]  $ which is $\left(
A^{\ast}\left(  \gamma_{0}\right)  ;\mathcal{C}^{\ast}\left(  \gamma
_{0}\right)  \right)  -$strictly invariant, then it is also $\left(  A^{\ast
}\left(  \gamma_{0}\right)  ;\left(  \mathcal{C}^{\ast}\left(  \gamma
_{0}\right)  +\mathcal{I}\left(  \gamma_{0}\right)  \right)  \Pi_{W}\right)  $
- invariant and, hence, $\left(  A^{\ast}\left(  \gamma_{0}\right)  ;\left(
\mathcal{C}^{\ast}\left(  \gamma_{0}\right)  +\mathcal{I}\left(  \gamma
_{0}\right)  \right)  \Pi_{Ker\left[  \left(  B^{0}\left(  \gamma_{0}\right)
\right)  ^{\ast}\right]  }\right)  $- invariant. Under the sufficient
condition described in the previous proposition, it follows that $W=\left\{
0\right\}  .$

(ii) The space $V^{0}$ is what is called the viability kernel of $Ker\left[
\left(  B^{0}\left(  \gamma_{0}\right)  \right)  ^{\ast}\right]  $ w.r.t. the
deterministic system (i.e. the larger set of initial data $y$ such that the
solution stays in $Ker\left[  \left(  B^{0}\left(  \gamma_{0}\right)  \right)
^{\ast}\right]  $). We emphasize that a solution starting from $V^{0}$ stays
in $V^{0}$ (prior to jump). In other words, this viability kernel is viable
w.r.t the deterministic system. However we cannot guarantee that this is still
the case w.r.t. the stochastic equation (the controls $u$ may not take their
values in $V^{0}$)$.$This possible non-viability is the reason we cannot (in
all generality) go any further to obtain a result similar to the one in
Proposition \ref{PropNec2}.

(iii) In the definition (\ref{defV0}), one can replace $\left(  \mathcal{C}%
^{\ast}\left(  \gamma_{0}\right)  +\mathcal{I}\left(  \gamma_{0}\right)
\right)  $ by $\mathcal{C}^{\ast}\left(  \gamma_{0}\right)  .$ However, in
general, this is no longer the case in $V^{i}$ when $i\geq1$.
\end{remark}

\section{Examples and Counterexamples, Equivalent
Criteria\label{SectionExpEquiv}}

We begin this section with some simple counterexamples. We examine several
cases. First, we show that the first necessary condition (given in Proposition
\ref{PropNec1}) may fail to imply the controllability of the associated
deterministic system, and, hence, of the stochastic one. In a second example
we show that it is possible that the first necessary condition (given in
Proposition \ref{PropNec1}) together with the Kalman condition for the
controllability of the deterministic system should fail to imply the second
necessary condition (of Proposition \ref{PropNec2}). The third example shows
that for multimode systems (i.e. non-constant coefficients), the necessary
condition of Proposition \ref{PropNec2} may fail to imply the condition given
in Proposition \ref{PropNec1} (and, thus, is not sufficient for the general
approximate null-controllability). We emphasize that in the last two examples
the Kalman condition for controllability of the associated deterministic
system is satisfied but it does not imply the approximate null-controllability
of the initial stochastic system.

Let us consider some simple examples in which the necessary condition given in
Proposition \ref{PropNec1} is satisfied and it fails to imply the approximate
null-controllability. To this purpose, we consider a two-dimensional state
space and a one-dimensional control space ($n=2,$ $d=1$). The systems will be
bimodal governed by $E=\left\{  0,1\right\}  ,$ the control will only act on
the first component of the state space $B^{0}\left(  0\right)  =B^{0}\left(
1\right)  =B^{0}:=\left(
\begin{array}
[c]{c}%
1\\
0
\end{array}
\right)  $. The control matrix $B_{t}$ is set to be constant (its equation has
$\beta=0$)$.$ The pure jump component will switch between $0$ and $1$ in
exponential time with the same parameter $\lambda\left(  0\right)
=\lambda\left(  1\right)  =1,$ $Q\left(  0,d\theta\right)  =\delta_{1}\left(
d\theta\right)  ,$ $Q\left(  1,d\theta\right)  =\delta_{0}\left(
d\theta\right)  $. \ One easily notes that $Ker\left(  \left(  B^{0}\right)
^{\ast}\right)  =\left\{  \left(
\begin{array}
[c]{c}%
0\\
y
\end{array}
\right)  :y\in%
\mathbb{R}
\right\}  .$ Let us consider $C\left(  i,j\right)  =\left(
\begin{array}
[c]{cc}%
0 & \frac{1}{2}\\
\frac{1}{2} & 0
\end{array}
\right)  $, for all $i,j\in E.$

\begin{example}
\label{Nec1notDetCtrl}We first consider $A\left(  0\right)  =A\left(
1\right)  =0\in%
\mathbb{R}
^{2\times2}.$ We get
\[
A^{\ast}\left(  \gamma\right)  -\lambda\left(  \gamma\right)  \int_{E}C^{\ast
}\left(  \gamma,\theta\right)  Q\left(  \gamma,d\theta\right)  =-\left(
\begin{array}
[c]{cc}%
0 & \frac{1}{2}\\
\frac{1}{2} & 0
\end{array}
\right)  ,\text{ for all }\gamma\in E.
\]
One notes that
\[
\left(  A^{\ast}\left(  \gamma\right)  -\lambda\left(  \gamma\right)  \int%
_{E}C^{\ast}\left(  \gamma,\theta\right)  Q\left(  \gamma,d\theta\right)
\right)  \left(
\begin{array}
[c]{c}%
0\\
y
\end{array}
\right)  =\left(
\begin{array}
[c]{c}%
-\frac{y}{2}\\
0
\end{array}
\right)  \in Ker\left(  \left(  B^{0}\right)  ^{\ast}\right)
\]
if and only if $y=0.$ Hence, the condition given in Proposition \ref{PropNec1}
is satisfied. The system (\ref{SDE0}) has the particular form
\[
\left\{
\begin{array}
[c]{l}%
dX_{s}^{\left(  x,y\right)  ,u}=d\left(
\begin{array}
[c]{c}%
x_{s}^{\left(  x,y\right)  ,u}\\
y_{s}^{\left(  x,y\right)  ,u}%
\end{array}
\right)  =\left(
\begin{array}
[c]{c}%
u_{s}\\
0
\end{array}
\right)  ds+\frac{1}{2}\int_{E}\left(
\begin{array}
[c]{c}%
y_{s-}^{\left(  x,y\right)  ,\gamma,u}\\
x_{s-}^{\left(  x,y\right)  ,\gamma,u}%
\end{array}
\right)  q\left(  ds,d\theta\right)  ,\\
X_{0}^{\left(  x,y\right)  ,u}=\left(
\begin{array}
[c]{c}%
x\\
y
\end{array}
\right)  \in%
\mathbb{R}
^{2}.
\end{array}
\right.
\]
However, one easily notes that $\ \mathbb{E}\left[  y_{s}^{\left(  x,y\right)
,u}\right]  =y$, for all $s>0$. Thus, by taking $y\neq0,$ the previous system
cannot be steered onto arbitrarily small neighborhoods of $\left(
\begin{array}
[c]{c}%
0\\
0
\end{array}
\right)  $.
\end{example}

Due to this example, the reader will get an obvious necessary condition (at
least for constant $A,B^{0})$ : in order for the system (\ref{SDE0}) to be
approximately controllable if $B^{0}\left(  \gamma\right)  =B^{0}$ and
$A\left(  \gamma\right)  =A,$ for all $\gamma\in E,$ the associated
deterministic system (corresponding to the expectation of $X$)
\[
\left\{
\begin{array}
[c]{l}%
dx_{s}^{x,\widetilde{u}}=\left(  Ax_{s}^{x,\widetilde{u}}+B^{0}\widetilde{u}%
_{s}\right)  ds,\text{ }s\geq0\\
x_{0}^{x,\widetilde{u}}=x,
\end{array}
\right.
\]
should be controllable or, equivalently, due to Kalman's criterion,
\[
Rank\left[  B^{0},\text{ }AB^{0},...,\text{ }A^{n-1}B^{0}\right]  =n.
\]
Naturally, this condition is not satisfied in the previous example.

At this point, the reader may be interested in knowing if this supplementary
necessary condition (the controllability of the associated deterministic
system) implies the approximate null-controllability for the stochastic system
(at least for constant $A,B^{0},C).$ The answer is negative as shown by the following.

\begin{example}
\label{Nec1+DetCtrlnotNec2}The coefficients $B^{0}$ and $C$ are taken as
before and the framework is similar. However, we take $A\left(  0\right)
=A\left(  1\right)  =A:=\left(
\begin{array}
[c]{cc}%
0 & 1\\
1 & 0
\end{array}
\right)  $. Then, we get
\[
A^{\ast}\left(  \gamma\right)  -\lambda\left(  \gamma\right)  \int_{E}C^{\ast
}\left(  \gamma,\theta\right)  Q\left(  \gamma,d\theta\right)  =\left(
\begin{array}
[c]{cc}%
0 & \frac{1}{2}\\
\frac{1}{2} & 0
\end{array}
\right)  ,\text{ for all }\gamma\in E.
\]
and the condition given in Proposition \ref{PropNec1} is again satisfied (as
in the previous case). It is clear that $AB^{0}=\left(
\begin{array}
[c]{c}%
0\\
1
\end{array}
\right)  $ and, thus, the Kalman condition is also satisfied.

Nevertheless, if one considers
\[
\xi:=\left(  -1\right)  ^{p\left(  \left[  0,T\right]  ,E\right)  }\left(
\begin{array}
[c]{c}%
0\\
e^{2T}%
\end{array}
\right)  ,Y_{t}:=\left(  -1\right)  ^{p\left(  \left[  0,t\right]  ,E\right)
}\left(
\begin{array}
[c]{c}%
0\\
e^{2t}%
\end{array}
\right)
\]
and $Z_{t}:=-2Y_{t-},$ $\left(  Y,Z\right)  $ is the unique solution of the
BSDE (\ref{BSDE0}) with final data $\xi.$ It is clear that $Y_{t}\in
Ker\left(  \left(  B^{0}\right)  ^{\ast}\right)  $ for all $t\in\left[
0,T\right]  .$ However, $Y_{0}=\left(
\begin{array}
[c]{c}%
0\\
1
\end{array}
\right)  $ $\neq\left(
\begin{array}
[c]{c}%
0\\
0
\end{array}
\right)  $ almost surely and, due to Theorem \ref{TheoremCar}, the system is
not approximately null-controllable. (Alternatively, one may want to note that
$V_{\infty}=V_{0}=Ker\left[  \left(  B^{0}\right)  ^{\ast}\right]  $ and apply
Proposition \ref{PropNec2} to get the same conclusion).
\end{example}

The examples we have seen so far show that, even for the case in which the
coefficients are constant, the necessary condition given in Proposition
\ref{PropNec1} might not be sufficient. In the same framework, even though we
have the exact controllability of the associated deterministic system, the
stochastic one might not be approximately null-controllable.

It is worth pointing out that, in particular cases (constant coefficients and
Poisson random measure-driven systems), the necessary condition given by
Proposition \ref{PropNec2} turns out to actually be sufficient. This will be
the purpose of the next Subsection. However, for general systems, this
condition may also fail to imply the approximate null-controllability. To
illustrate this assertion, we consider the following switch system.

\begin{example}
\label{Nec2+DetCtrlnotNec1}We let $n=3,$ $d=1$, $E=\left\{  0,1\right\}  ,$
$B^{0}\left(  0\right)  =B^{0}\left(  1\right)  =B^{0}:=\left(
\begin{array}
[c]{c}%
1\\
0\\
0
\end{array}
\right)  $, $\beta=0,$ $\lambda\left(  0\right)  =\lambda\left(  1\right)
=1,$ $Q\left(  0,d\theta\right)  =\delta_{1}\left(  d\theta\right)  ,$
$Q\left(  1,d\theta\right)  =\delta_{0}\left(  d\theta\right)  $. \ Moreover,
we denote by $e_{1}=\left(
\begin{array}
[c]{c}%
1\\
0\\
0
\end{array}
\right)  ,e_{2}=\left(
\begin{array}
[c]{c}%
0\\
1\\
0
\end{array}
\right)  ,e_{3}=\left(
\begin{array}
[c]{c}%
0\\
0\\
1
\end{array}
\right)  .$ One easily notes that $Ker\left(  \left(  B^{0}\right)  ^{\ast
}\right)  =span\left(  e_{2},e_{3}\right)  .$ We consider $A\left(  0\right)
:=\left(
\begin{array}
[c]{ccc}%
0 & 0 & 0\\
1 & 0 & 0\\
0 & 1 & 0
\end{array}
\right)  ,$ $A\left(  1\right)  :=\left(
\begin{array}
[c]{ccc}%
0 & 0 & 0\\
0 & 0 & 1\\
1 & 0 & 0
\end{array}
\right)  ,$ $C\left(  0,1\right)  :=\left(
\begin{array}
[c]{ccc}%
0 & 0 & 0\\
0 & 0 & 0\\
0 & 1 & 0
\end{array}
\right)  $, $C\left(  1,0\right)  :=\left(
\begin{array}
[c]{ccc}%
0 & 0 & 0\\
0 & 0 & 1\\
0 & 0 & 0
\end{array}
\right)  .$ The reader is invited to note that $Rank\left(  B^{0}\text{
}A\left(  0\right)  B^{0}\text{ }A\left(  0\right)  ^{2}B^{0}\right)  =3$ and,
thus, the associated deterministic system is controllable. The same assertion
holds true for $A\left(  1\right)  $ replacing $A\left(  0\right)  $.
Moreover, $A\left(  0\right)  ^{\ast}\left(  ye_{2}+ze_{3}\right)
=ye_{1}+ze_{2},$ $C\left(  0,1\right)  ^{\ast}\left(  ye_{2}+ze_{3}\right)
=ze_{2}.$ Hence, the largest subspace of $Ker\left(  \left(  B^{0}\right)
^{\ast}\right)  $ which is $\left(  A\left(  0\right)  ^{\ast};C\left(
0,1\right)  ^{\ast}\Pi_{Ker\left(  \left(  B^{0}\right)  ^{\ast}\right)
}\right)  $-invariant is $span\left(  e_{3}\right)  $. Similarly, the largest
subspace of $Ker\left(  \left(  B^{0}\right)  ^{\ast}\right)  $ which is
$\left(  A\left(  1\right)  ^{\ast};C\left(  1,0\right)  ^{\ast}%
\Pi_{Ker\left(  \left(  B^{0}\right)  ^{\ast}\right)  }\right)  $-invariant is
$span\left(  e_{2}\right)  $. We deduce that the space $V_{\infty}$ appearing
in Proposition \ref{PropNec2} is reduced to $\left\{  0\right\}  .$ However,
\[
A^{\ast}\left(  0\right)  -\lambda\left(  0\right)  \int_{E}C^{\ast}\left(
0,\theta\right)  Q\left(  0,d\theta\right)  =\left(
\begin{array}
[c]{ccc}%
0 & 1 & 0\\
0 & 0 & 0\\
0 & 0 & 0
\end{array}
\right)
\]
and the space $span\left(  e_{3}\right)  $ is $A^{\ast}\left(  0\right)
-\lambda\left(  0\right)  \int_{E}C^{\ast}\left(  0,\theta\right)  Q\left(
0,d\theta\right)  $-invariant. Similar assertions hold true for $\gamma
_{0}=1.$ It follows that, although the necessary condition in \ref{PropNec2}
holds true, the necessary condition given by Proposition \ref{PropNec1} is not
satisfied. Hence, the system is not approximately null-controllable.
\end{example}

\subsection{The Constant Coefficients Case\label{SubsectionConstantCoeff}}

Throughout the subsection, we fix $\gamma_{0}\in E$ and drop the dependency of
$\gamma_{0}$ in $\mathbb{P}^{0,\gamma_{0}}$. We consider the following
particular form of the system (\ref{SDE0}).%
\begin{equation}
\left\{
\begin{array}
[c]{l}%
dX_{s}^{x,u}=\left[  AX_{s}^{x,u}+Bu_{s}\right]  ds+\int_{E}C\left(
\theta\right)  X_{s-}^{x,u}q\left(  ds,d\theta\right)  ,\text{ }s\geq0,\\
X_{0}^{x,u}=x,
\end{array}
\right.  \label{SDE0'}%
\end{equation}
where $A\in%
\mathbb{R}
^{n\times n},$ $B\in%
\mathbb{R}
^{n\times d}$ are fixed and $C\left(  \theta\right)  \in%
\mathbb{R}
^{n\times n},$ for all $\theta\in E$ such that $\lambda\left(  \gamma\right)
Q\left(  \gamma,d\theta\right)  $ is independent of $\gamma\in E.$ In this
case, $q$ corresponds to a (compensated) Poisson random measure. We let $p$
denote the Poisson random measure and $\nu\left(  d\theta\right)
:=\lambda\left(  \gamma\right)  Q\left(  \gamma,d\theta\right)  $ be its
L\'{e}vy measure in this framework. Let us point out that more general
$\sigma$-finite L\'{e}vy measures satisfying a suitable second order moment
condition can be considered and the arguments are identical. However, for
coherence reasons, we work with the (finite) measure given before. We assume
that $C\in\mathbb{L}^{2}\left(  E,\mathcal{B}\left(  E\right)  ,\nu;%
\mathbb{R}
^{n\times n}\right)  $. As we have already hinted before, we interpret the
system (\ref{BSDE0}) as a controlled, forward one. In this case, the system
(\ref{SDE'}) takes the particular form%
\begin{equation}
\left\{
\begin{array}
[c]{l}%
dY_{t}^{y,v}=\left(  -A^{\ast}Y_{t}^{y,v}-\int_{E}C^{\ast}\left(
\theta\right)  v_{t}\left(  \theta\right)  \nu\left(  d\theta\right)  \right)
dt+\int_{E}v_{t}\left(  \theta\right)  q\left(  dt,d\theta\right)  ,\text{
}t\geq0,\\
Y_{0}^{y,v}=y\in%
\mathbb{R}
^{n}.
\end{array}
\right.  \label{SDE''}%
\end{equation}

Let us point out that, in this case, Proposition \ref{PropNec2} yields that,
whenever the system (\ref{SDE0'}) is approximately null-controllable in time
$T>0,$ the largest subspace $V_{0}\subset KerB^{\ast}$ which is $\left(
A^{\ast};\mathcal{C}^{\ast}\Pi_{V_{0}}\right)  $-invariant is reduced to
$\left\{  0\right\}  $ (i.e. the largest subspace of $KerB^{\ast}$ which is
$\left(  A^{\ast};\mathcal{C}^{\ast}\right)  -$strictly invariant is reduced
to $\left\{  0\right\}  $)$.$ We recall that the linear operator
$\mathcal{C}^{\ast}:\mathbb{L}^{2}\left(  E,\mathcal{B}\left(  E\right)  ,\nu;%
\mathbb{R}
^{n}\right)  \longrightarrow%
\mathbb{R}
^{n}$ is given by%
\[
\mathcal{C}^{\ast}\phi=\int_{E}C^{\ast}\left(  \theta\right)  \phi\left(
\theta\right)  \nu\left(  d\theta\right)  ,\text{ for all }\phi\in
\mathbb{L}^{2}\left(  E,\mathcal{B}\left(  E\right)  ,\nu;%
\mathbb{R}
^{n}\right)  \text{.}%
\]
We wish to prove that this condition is also sufficient in order to have
approximate null-controllability. In fact, in this context, we show that not
only do we have approximate null-controllability, but we actually get
approximate controllability.

A careful look at the proof of Proposition \ref{PropSuf1} and Remark
\ref{ViabKernel} (ii) leads to considering%
\[
Viab_{loc}\left(  KerB^{\ast}\right)  =\left\{
\begin{array}
[c]{c}%
y\in KerB^{\ast}:\exists T>0,\text{ }v\in\mathcal{L}^{2}\left(  p;%
\mathbb{R}
^{n}\right)  \text{ s.t. }\\
Y^{y,v}\in KerB^{\ast},\text{ }\mathbb{P\otimes}Leb-a.s.\text{ on }%
\Omega\times\left[  0,T\right]
\end{array}
\right\}  .
\]
When $y\in Viab_{loc}\left(  KerB^{\ast}\right)  ,$ one finds a solution which
is constrained to $KerB^{\ast}.$ The control problems with state constraints
can be transformed into unconstrained ones by adding a penalty when the
constraint is not satisfied. Heuristically, in this context, one may associate
to $y\in%
\mathbb{R}
^{n}$ a cost which involves a penalty term $N\mathbb{E}\left[  \int_{0}%
^{T}\left\vert \Pi_{\left[  KerB^{\ast}\right]  ^{\bot}}\left(  Y_{t}%
^{y,v}\right)  \right\vert ^{2}dt\right]  $ and allow $N\rightarrow\infty$.
Since this type of cost belongs to the class of LQ (linear-quadratic) control
problems, it can be addressed with the theory of Riccati equations. We
consider the sequence of (deterministic) Riccati equations%
\[
\left\{
\begin{array}
[c]{l}%
dK_{t}^{N}=\left[
\begin{array}
[c]{c}%
-K_{t}^{N}A^{\ast}-AK_{t}^{N}+N\Pi_{\left[  KerB^{\ast}\right]  ^{\bot}}\\
-K_{t}^{N}\int_{E}C^{\ast}\left(  \theta\right)  \left(  I+K_{t}^{N}\right)
^{-1}C\left(  \theta\right)  \nu\left(  d\theta\right)  K_{t}^{N}%
\end{array}
\right]  dt,\text{ }t\geq0,\\
K_{0}^{N}=0.
\end{array}
\right.
\]
The proof for the existence and uniqueness of the solution is quite standard.
It relies on successive iterations (see \cite[Chap. 6, Cor. 2.10 and Prop.
2.12 ]{yong_zhou_99}) and will be omitted from the present paper. Obviously,
the sequence $\left(  K^{N}\right)  _{N\geq1}$ is non-decreasing in the family
$\mathcal{S}_{n}^{+}$ of symmetric, positive semi-definite matrix.

\begin{remark}
In fact, in the theory of stochastic control, one would be led to consider the
following BSDE%
\[
\left\{
\begin{array}
[c]{l}%
dR_{t}^{N}\\
=\left(  R_{t}^{N}A^{\ast}+AR_{t}^{N}-N\Pi_{\left[  Ker\left(  B^{\ast
}\right)  \right]  ^{\bot}}\right)  dt+\int_{E}H_{t}^{N}\left(  \theta\right)
q\left(  dt,d\theta\right) \\
+\int_{E}\left(  f_{t}^{N}\left(  \theta\right)  \right)  ^{\ast}\left(
I+R_{t}^{N}+H_{t}^{N}\left(  \theta\right)  \right)  ^{-1}f_{t}^{N}\left(
\theta\right)  \nu\left(  d\theta\right)  dt,\\
\text{where }f_{t}^{N}\left(  \theta\right)  :=\left(  C\left(  \theta\right)
R_{t}^{N}-H_{t}^{N}\left(  \theta\right)  \right)  ,\text{ }t\in\left[
0,T\right]  .\\
R_{T}^{N}=0,\text{ }I+R_{t}^{N}+H_{t}^{N}\left(  \theta\right)  >0,\text{ for
almost all }t\in\left[  0,T\right]  .
\end{array}
\right.
\]
However, in our particular setting, the solution of the BSDE is got by setting
$R_{t}^{N}=K_{T-t}^{N}$ and $H_{t}^{N}=0.$
\end{remark}

We get the following.

\begin{proposition}
\label{PropViablocKerB*}If $y\in Viab_{loc}\left(  KerB^{\ast}\right)  ,$
then, there exists $T>0$ and $v\in\mathcal{L}^{2}\left(  p;%
\mathbb{R}
^{n}\right)  $ such that $Y_{t}^{y,v}\in Viab_{loc}\left(  KerB^{\ast}\right)
, $ $\mathbb{P\otimes}Leb-a.s.$ on $\Omega\times\left[  0,T\right]  .$
\end{proposition}

\begin{proof}
\underline{Step 1.} We claim that
\[
Viab_{loc}\left(  KerB^{\ast}\right)  =\mathbb{K}:=\left\{  \widetilde{y}\in%
\mathbb{R}
^{n}:\exists\widetilde{T}>0\text{ s.t. }\underset{N\rightarrow\infty}{\lim
}\left\langle K_{\widetilde{T}}^{N}\widetilde{y},\widetilde{y}\right\rangle
<\infty\right\}  .
\]

Using the definition of $Viab_{loc}\left(  KerB^{\ast}\right)  $, we are led
to consider, for $y\in Viab_{loc}\left(  KerB^{\ast}\right)  ,$ some $T>0$ and
$v\in\mathcal{L}^{2}\left(  p;%
\mathbb{R}
^{n}\right)  $ such that $Y^{y,v}\in KerB^{\ast},$ $\mathbb{P\otimes}Leb-a.s.$
on $\Omega\times\left[  0,T\right]  $. \ It\^{o}'s formula applied to
$\left\langle K_{T-s}^{N}Y_{s}^{y,v},Y_{s}^{y,v}\right\rangle $ for
$s\in\left[  t,T\right]  $ yields%
\begin{align*}
&  \mathbb{E}\left[  \left\langle K_{T-t}^{N}Y_{t}^{y,v},Y_{t}^{y,v}%
\right\rangle \right] \\
&  =\mathbb{E}\left[  \int_{t}^{T}\int_{E}\left\vert v_{s}\left(
\theta\right)  \right\vert ^{2}\nu\left(  d\theta\right)  ds\right]
+N\mathbb{E}\left[  \int_{t}^{T}\left\vert \Pi_{\left[  KerB^{\ast}\right]
^{\bot}}\left(  Y_{s}^{y,v}\right)  \right\vert ^{2}ds\right] \\
&  -\mathbb{E}\left[  \underset{\left[  t,T\right]  \times E}{\int}\left\vert
\left(  I+K_{T-s}^{N}\right)  ^{\frac{1}{2}}v_{s}\left(  \theta\right)
-\left(  I+K_{T-s}^{N}\right)  ^{-\frac{1}{2}}C\left(  \theta\right)
K_{T-s}^{N}Y_{s-}^{y,v}\right\vert ^{2}\nu\left(  d\theta\right)  ds\right]  .
\end{align*}
For $t=0$, it follows that
\[
\left\langle K_{T}^{N}y,y\right\rangle \leq\mathbb{E}\left[  \int_{0}^{T}%
\int_{E}\left\vert v_{s}\left(  \theta\right)  \right\vert ^{2}\nu\left(
d\theta\right)  ds\right]  ,\text{ for all }N\geq1,
\]
hence $y\in\mathbb{K}.$ Conversely, if $y\in\mathbb{K}$,$\ $and $T>0$ is such
that $\underset{N\rightarrow\infty}{\lim}\left\langle K_{T}^{N}%
y,y\right\rangle =:c<\infty$, one takes the feedback control sequence
\[
v_{s}^{N}\left(  \theta\right)  :=\left(  I+K_{T-s}^{N}\right)  ^{-1}C\left(
\theta\right)  K_{T-s}^{N}Y_{s-}^{y,v},
\]
to get
\[
\mathbb{E}\left[  \int_{0}^{T}\int_{E}\left\vert v_{s}^{N}\left(
\theta\right)  \right\vert ^{2}\nu\left(  d\theta\right)  ds\right]
+N\mathbb{E}\left[  \int_{0}^{T}\left\vert \Pi_{\left[  KerB^{\ast}\right]
^{\bot}}\left(  Y_{s}^{y,v^{N}}\right)  \right\vert ^{2}ds\right]  \leq c.
\]
Hence, one finds a (sub)sequence $\left(  v^{N},Y^{y,v^{N}}\right)  $ weakly
converging to some $\left(  v^{\ast},Y^{\ast}\right)  \in\mathcal{L}%
^{2}\left(  p;%
\mathbb{R}
^{n}\right)  \times\mathbb{L}^{2}\left(  \left[  0,T\right]  \times
\Omega,Prog^{0},Leb\otimes\mathbb{P};%
\mathbb{R}
^{n}\right)  .$ Linearity arguments allow one to identify $Y^{\ast
}=Y^{y,v^{\ast}}$ and the semicontinuity of the $\mathbb{L}^{2}$-norm w.r.t.
the weak topology to infer that $Y^{y,v^{\ast}}\in KerB^{\ast},$
$\mathbb{P\otimes}Leb-a.s.$ on $\Omega\times\left[  0,T\right]  $. We have,
thus, shown that $y\in Viab_{loc}\left(  KerB^{\ast}\right)  .$

\underline{Step 2.} We proceed as in the first part of the previous step. Let
us fix $y\in Viab_{loc}\left(  KerB^{\ast}\right)  ,$ $T>0$ and $v\in
\mathcal{L}^{2}\left(  p;%
\mathbb{R}
^{n}\right)  $ such that $Y^{y,v}\in KerB^{\ast},$ $\mathbb{P\otimes}Leb-a.s.$
on $\Omega\times\left[  0,T\right]  $. Whenever $t\leq T$,
\[
\mathbb{E}\left[  \left\langle K_{T-t}^{N}Y_{t}^{y,v},Y_{t}^{y,v}\right\rangle
\right]  \leq\mathbb{E}\left[  \int_{0}^{T}\int_{E}\left\vert v_{s}\left(
\theta\right)  \right\vert ^{2}\nu\left(  d\theta\right)  ds\right]  .
\]
Using Fatou's Lemma, one gets that
\[
\mathbb{E}\left[  \underset{N\rightarrow\infty}{\lim\inf}\left\langle
K_{T-t}^{N}Y_{t}^{y,v},Y_{t}^{y,v}\right\rangle \right]  \leq\mathbb{E}\left[
\int_{0}^{T}\int_{E}\left\vert v_{s}\left(  \theta\right)  \right\vert ^{2}%
\nu\left(  d\theta\right)  ds\right]  .
\]
Hence, $\underset{N\rightarrow\infty}{\lim\inf}\left\langle K_{T-t}^{N}%
Y_{t}^{y,v},Y_{t}^{y,v}\right\rangle <\infty,$ $\mathbb{P-}$a.s. and, using
the first step, $Y_{t}^{y,v}\in Viab_{loc}\left(  KerB^{\ast}\right)  ,$
$\mathbb{P-}$a.s. The conclusion follows by recalling that $Y_{\cdot}^{y,v}$
is c\`{a}dl\`{a}g.
\end{proof}

\begin{remark}
The definition of $\mathbb{K}$ can, equivalently, be given as
\[
\mathbb{K}=\left\{  \widetilde{y}\in%
\mathbb{R}
^{n}:\exists\widetilde{T}>0\text{ s.t. }\underset{N\rightarrow\infty}{\lim
}\left\langle K_{t}^{N}\widetilde{y},\widetilde{y}\right\rangle <\infty,\text{
for all }0\leq t\leq\widetilde{T}\right\}  .
\]
In particular, this implies that $Viab_{loc}\left(  KerB^{\ast}\right)
=\mathbb{K}$ is a linear subspace of $KerB^{\ast}.$
\end{remark}

As a consequence, we get the following criterion for the approximate and
approximate null-controllability in the case when the coefficients are constant.

\begin{criterion}
\label{CritEquiv}The system (\ref{SDE0'}) is approximately null-controllable
iff it is approximately controllable. The necessary and sufficient condition
for approximate controllability is that the largest subspace $V_{0}\subset
KerB^{\ast}$ which is $\left(  A^{\ast};\mathcal{C}^{\ast}\Pi_{V_{0}}\right)
$-invariant be reduced to $\left\{  0\right\}  $.
\end{criterion}

\begin{proof}
Approximate controllability implies approximate null-controllability, which,
by Proposition \ref{PropNec2}, implies that the largest subspace $V_{0}\subset
KerB^{\ast}$ which is $\left(  A^{\ast};\mathcal{C}^{\ast}\Pi_{V_{0}}\right)
$-invariant is reduced to $\left\{  0\right\}  $.

For the converse we proceed as in the proof of Proposition \ref{PropSuf1}. Let
us just hint the modifications needed to this proof. One begins by applying
It\^{o}'s formula to $\left(  I-\Pi_{Viab_{loc}\left(  KerB^{\ast}\right)
}\right)  Y_{t}^{y,v}$ (instead of $\left(  I-\Pi_{Ker\left(  B_{t}^{\ast
}\right)  }\right)  Y_{t}^{y,v}$). Whenever $B^{\ast}Y_{t}^{y,v}=0,$
$\mathbb{P\otimes}Leb$ almost everywhere on $\Omega\times\left[  0,T\right]
,$ it follows, due to the previous result, that $Y_{t}^{y,v}\in Viab_{loc}%
\left(  KerB^{\ast}\right)  $. Hence, reasoning as in Proposition
\ref{PropSuf1},
\[
0=\mathbb{E}\left[  \int_{0}^{t}\int_{E}\left\vert \left(  I-\Pi
_{Viab_{loc}\left(  KerB^{\ast}\right)  }\right)  v_{s}\left(  \theta\right)
\right\vert ^{2}\nu\left(  d\theta\right)  ds\right]  .
\]
Then $v_{t}\left(  \theta\right)  \in Viab_{loc}\left(  KerB^{\ast}\right)  ,$
$\mathbb{P}\otimes Leb\otimes\nu$ almost everywhere on $\Omega\times\left[
0,T\right]  \times E.$ Reasoning as in Proposition \ref{PropSuf1} (but this
time globally, not only prior to the first jump time), we get $Y_{t}^{y,v}\in
V$, almost everywhere on $\Omega\times\left[  0,T\right]  .$ Here $V\subset$
$KerB^{\ast}$ is the linear subspace
\begin{equation}
V=\left\{
\begin{array}
[c]{c}%
y^{\prime}\in KerB^{\ast}:\exists u\in\mathbb{L}^{2}\left(  E,\mathcal{B}%
\left(  E\right)  ,\nu;Viab_{loc}\left(  KerB^{\ast}\right)  \right)  \text{
s.t.}\\
A^{\ast}y^{\prime}+\mathcal{C}^{\ast}u\in V.
\end{array}
\right\}
\end{equation}
We recall that $y$ is arbitrary in $Viab_{loc}\left(  KerB^{\ast}\right)  $ to
get $Viab_{loc}\left(  KerB^{\ast}\right)  \subset V$ and, thus, $V$ is
included in the largest subspace $V_{0}$ of $KerB^{\ast}$ which is $\left(
A^{\ast};\mathcal{C}^{\ast}\Pi_{V_{0}}\right)  $-invariant. By assumption,
$V_{0}=\left\{  0\right\}  .$ Hence, $Y_{t}^{y,v}=0,$ almost everywhere on
$\Omega\times\left[  0,T\right]  $ which implies approximate controllability
of (\ref{SDE0'}). The proof is now complete.
\end{proof}

\begin{remark}
This result is the counterpart of \cite{Buckdahn_Quincampoix_Tessitore_2006}
in this jump-framework. As in the cited paper, it turns out that the two
notions of approximate controllability coincide. Moreover, their
characterization is given, similar to
\cite{Buckdahn_Quincampoix_Tessitore_2006}, in terms of strict invariance.
However, while in the Brownian setting the operators are finite, here we have
an operator $\mathcal{C}^{\ast}$ acting on $\mathbb{L}^{2}.$ Except the case
when $E$ is finite, algorithms are more difficult to implement in this setting.
\end{remark}

Since, in this particular framework, we are able to completely characterize
the (equivalent) approximate controllability properties, we can produce a
simple example in which the condition given by Proposition \ref{PropSuf1} is
not satisfied, yet the system is approximately controllable. Hence, although
sufficient, this condition is not (always) necessary in the study of
approximate controllability.

\begin{example}
\label{CtrlnotSuf1}We let $n=3,$ $d=1,$ $E=\left\{  0,1\right\}  ,$
$B:=\left(
\begin{array}
[c]{c}%
1\\
0\\
0
\end{array}
\right)  $, $\beta=0,$ $\lambda\left(  0\right)  =\lambda\left(  1\right)
=1,$ $Q\left(  0,d\theta\right)  =\delta_{1}\left(  d\theta\right)  ,$
$Q\left(  1,d\theta\right)  =\delta_{0}\left(  d\theta\right)  $. \ We denote,
as before, by $e_{1}=\left(
\begin{array}
[c]{c}%
1\\
0\\
0
\end{array}
\right)  ,e_{2}=\left(
\begin{array}
[c]{c}%
0\\
1\\
0
\end{array}
\right)  ,e_{3}=\left(
\begin{array}
[c]{c}%
0\\
0\\
1
\end{array}
\right)  .$ One easily notes that $Ker\left(  B^{\ast}\right)  =span\left(
e_{2},e_{3}\right)  .$ We let $A:=\left(
\begin{array}
[c]{ccc}%
0 & 0 & 0\\
1 & 0 & 0\\
0 & 1 & 0
\end{array}
\right)  ,$ $C\left(  0\right)  =C\left(  1\right)  =C:=\left(
\begin{array}
[c]{ccc}%
1 & 0 & 0\\
0 & 0 & 0\\
0 & 0 & 0
\end{array}
\right)  $. The reader is invited to note that $A^{\ast}\left(  ye_{2}%
+ze_{3}\right)  =ye_{1}+ze_{2}$, $C^{\ast}\left(  ye_{2}+ze_{3}\right)  =0 $
and $\left(  C^{\ast}+I\right)  \left(  ye_{2}+ze_{3}\right)  =ye_{2}+ze_{3}
$. Hence, the largest subspace $V_{0}\subset KerB^{\ast}$ which is $\left(
A^{\ast};\mathcal{C}^{\ast}\Pi_{V_{0}}\right)  $-invariant (i.e $A^{\ast}$-
invariant) is $V_{0}=\left\{  0\right\}  .$ It follows that the system
(\ref{SDE0'}) is approximately controllable. However, the space $span\left(
e_{3}\right)  $ is $\left(  A^{\ast};\left(  \mathcal{C}^{\ast}+\mathcal{I}%
\right)  \Pi_{KerB^{\ast}}\right)  $-invariant and, thus, the condition of
Proposition \ref{PropSuf1} is not satisfied.
\end{example}

\subsection{The Continuous Switching Systems\label{SubsectionContSwitching}}

The next class of systems we address is the case in which the dynamics are
only switched but no jump takes place on the $X$ component. In other words,
following the behavior of $\gamma,$ the matrix $A$ is different $\left(
A\left(  \gamma\right)  \right)  ,$ but a change in the mode $\gamma$ does not
induce an instantaneous change of the process $X$. This corresponds to the
particular choice $C=0.$ In this case, we consider a switch piecewise linear
system%
\begin{equation}
\left\{
\begin{array}
[c]{l}%
dX_{s}^{x,u}=\left[  A\left(  \gamma_{s}\right)  X_{s}^{x,u}+B_{s}%
u_{s}\right]  ds,\text{ }s\geq0,\\
X_{0}^{x,u}=x,
\end{array}
\right.  \label{SDE0''}%
\end{equation}
The reader is invited to note that whenever spikes appear, they do not modify
the current $X$ component but merely the vector field.

The main result in this framework is the following criterion.

\begin{criterion}
\label{CritContSwitch}The system (\ref{SDE0''}) is approximately
null-controllable if and only if the largest subspace of $Ker\left[  \left(
B^{0}\right)  ^{\ast}\right]  $ which is $A^{\ast}\left(  \gamma_{0}\right)  $
- invariant is reduced to the trivial subspace $\left\{  0\right\}  $ for all
$\gamma_{0}\in E.$
\end{criterion}

Let us give a proof that can be adapted for piecewise linear systems with
mutual influence (i.e. in the case when the jump mechanism of the process
$\gamma$ equally depends on $X$)$.$ We refer the interested reader to
\cite{Davis_86}, \cite{Soner86_2}, \cite{Dempster_89}, \cite{G7} for the exact
construction. We will not present the rigorous framework but only give the
proof in our case. The main difficulty in the general framework is due to the
fact that the filtration is not fixed prior to the construction of $X$ and
that $\gamma$ itself would depend on the control parameter through $X.$ In
order to get a Markov process, one would have to use piecewise open-loop
control policies (i.e., between consecutive jump times $T_{n}$ and $T_{n+1},$
the control would be of type $u_{n}\left(  t-T_{n},\gamma_{T_{n}},X_{T_{n}%
}\right)  $ and would not depend on $\left(  T_{k},\gamma_{T_{k}},X_{T_{k}%
}\right)  _{k\leq n-1}$). \ In our proof of sufficiency, we shall exhibit
stabilizing controls in the open-loop feedback form.

Let us assume that the initial matrix $B^{0}$ is independent of $\gamma$ and
that $A\left(  \gamma\right)  $ are self-adjoint and commute with $B^{0}$%
\begin{equation}
A\left(  \gamma\right)  =A^{\ast}\left(  \gamma\right)  \text{ and }A\left(
\gamma\right)  B^{0}\left(  B^{0}\right)  ^{\ast}=B^{0}\left(  B^{0}\right)
^{\ast}A\left(  \gamma\right)  , \label{commuting}%
\end{equation}
for all $\gamma\in E.$ This assumption is a technical one and implies, without
any further considerations, the commutativity of Gramians and, hence, the
positiveness of products.

\begin{proof}
The necessary condition follows from Proposition \ref{PropNec1} since, in this
framework, $C=0.$

For the converse, let us fix $\gamma_{0}\in E,$ $x_{0}\in%
\mathbb{R}
^{n}$, a time horizon $T>0$ and (for the time being) a discretization step
$N>0$. We also drop the dependency of $\gamma_{0}$ in $\mathbb{P}%
^{0,\gamma_{0}}.$ Since the largest subspace of $Ker\left[  \left(
B^{0}\right)  ^{\ast}\right]  $ which is $A^{\ast}\left(  \gamma\right)  $ -
invariant is reduced to the trivial subspace $\left\{  0\right\}  ,$ it
follows that the controllability Gramian
\[
\mathcal{G}\left(  \gamma,t\right)  :=\int_{0}^{t}e^{A\left(  \gamma\right)
\left(  t-s\right)  }B^{0}\left(  B^{0}\right)  ^{\ast}e^{A^{\ast}\left(
\gamma\right)  \left(  t-s\right)  }ds
\]
is invertible for all $t>0$. Due to (\ref{commuting}), one has, for every
$0\leq t\leq t^{\prime},$%
\begin{align*}
&  e^{A\left(  \gamma\right)  \left(  t-s\right)  }B^{0}\left(  B^{0}\right)
^{\ast}e^{A^{\ast}\left(  \gamma\right)  \left(  t-s\right)  }e^{A\left(
\gamma\right)  \left(  t^{\prime}-r\right)  }B^{0}\left(  B^{0}\right)
^{\ast}e^{A^{\ast}\left(  \gamma\right)  \left(  t^{\prime}-r\right)  }\\
&  =e^{A\left(  \gamma\right)  \left(  t^{\prime}-r\right)  }B^{0}\left(
B^{0}\right)  ^{\ast}e^{A^{\ast}\left(  \gamma\right)  \left(  t^{\prime
}-r\right)  }e^{A\left(  \gamma\right)  \left(  t-s\right)  }B^{0}\left(
B^{0}\right)  ^{\ast}e^{A^{\ast}\left(  \gamma\right)  \left(  t-s\right)  }.
\end{align*}
Integrating this equality with respect to the Lebesgue measure for $\left(
s,r\right)  \in\left[  0,t\right]  \times\left[  0,t^{\prime}\right]  ,$ one
gets%
\[
\mathcal{G}\left(  \gamma,t\right)  \mathcal{G}\left(  \gamma,t^{\prime
}\right)  =\mathcal{G}\left(  \gamma,t^{\prime}\right)  \mathcal{G}\left(
\gamma,t\right)  .
\]
Whenever $t^{\prime}>0,$ one also has%

\begin{equation}
\mathcal{G}\left(  \gamma,t\right)  \mathcal{G}^{-1}\left(  \gamma,t^{\prime
}\right)  =\mathcal{G}^{-1}\left(  \gamma,t^{\prime}\right)  \mathcal{G}%
\left(  \gamma,t\right)  . \label{commuteG}%
\end{equation}
Hence, in this commutating framework, $\mathcal{G}\left(  \gamma,t\right)
\mathcal{G}^{-1}\left(  \gamma,t^{\prime}\right)  $ is also positive definite
(except when $t=0$ and the product is $0).$ We define the (square norm)
minimal action (in time $\frac{T}{N}$)
\begin{align*}
u^{0}\left(  \gamma,y,t\right)   &  :=\left(  -e^{-\beta\gamma t}\left(
B^{0}\right)  ^{\ast}e^{A^{\ast}\left(  \gamma\right)  \left(  \frac{T}%
{N}-t\right)  }\mathcal{G}^{-1}\left(  \gamma,\frac{T}{N}\right)  e^{A\left(
\gamma\right)  \frac{T}{N}}y\right)  1_{t\leq\frac{T}{N}}\\
&  +01_{t>\frac{T}{N}},
\end{align*}
for every $\gamma\in E,$ $y\in%
\mathbb{R}
^{n},$ $t\geq0$ and consider the trajectory $X$ associated to the control
sequence $\left(  u^{n}=\left\{
\begin{array}
[c]{l}%
u^{0}\text{, if }N\geq n\geq1,\\
0\text{ \ , otherwise}%
\end{array}
\right.  \right)  _{n\geq1}$, i.e. to the control%
\[
u_{t}:=u^{1}\left(  \gamma_{0},x_{0},t\right)  1_{0\leq t\leq T_{1}}%
+\sum_{n\geq2}u^{n}\left(  \gamma_{T_{n-1}},X_{T_{n-1}},t-T_{n-1}\right)
1_{T_{n-1}<t\leq T_{n}}.
\]
(The trajectory is constructed $\omega$-wise between any two consecutive jump
times. Obviously, with this kind of construction, $u$ can also act on $\gamma
$, on $\lambda$ and on $Q$. For more detailed construction, the reader is
referred to \cite{Davis_86}, \cite{Soner86_2}, \cite{Dempster_89}, \cite{G7},
etc.). Then $X_{T}^{x_{0},u}=0$ on $\left(  T_{1}\geq\frac{T}{N}\right)
\cup\left(  T_{2}-T_{1}\geq\frac{T}{N}\right)  \cup...\cup\left(
T_{N}-T_{N-1}\geq\frac{T}{N}\right)  .$ Up to $T_{1}$, the explicit solution
of our equation with mode $\gamma_{0}$ and initial data $x_{0}$ is given by
\begin{align*}
&  X_{t}^{x_{0},u}\\
&  =\left(  I-\mathcal{G}\left(  \gamma_{0},t\right)  e^{A^{\ast}\left(
\gamma_{0}\right)  \left(  \frac{T}{N}-t\right)  }\mathcal{G}^{-1}\left(
\gamma_{0},\frac{T}{N}\right)  e^{A\left(  \gamma_{0}\right)  \left(  \frac
{T}{N}-t\right)  }\right)  e^{A\left(  \gamma_{0}\right)  t}x_{0}1_{t\leq
\frac{T}{n}\wedge T_{1}}.
\end{align*}
In view of (\ref{commuteG}) and (\ref{commuting}), we get that
\begin{align*}
&  \mathcal{G}\left(  \gamma_{0},t\right)  e^{A^{\ast}\left(  \gamma
_{0}\right)  \left(  \frac{T}{N}-t\right)  }\mathcal{G}^{-1}\left(  \gamma
_{0},\frac{T}{N}\right)  e^{A\left(  \gamma_{0}\right)  \left(  \frac{T}%
{N}-t\right)  }\\
&  =e^{A^{\ast}\left(  \gamma_{0}\right)  \left(  \frac{T}{N}-t\right)
}\mathcal{G}\left(  \gamma_{0},t\right)  \mathcal{G}^{-1}\left(  \gamma
_{0},\frac{T}{N}\right)  e^{A\left(  \gamma_{0}\right)  \left(  \frac{T}%
{N}-t\right)  }%
\end{align*}
is positive definite (or $0$ is $t=0$). On the other hand, since
\[
\mathcal{G}\left(  \gamma_{0},\frac{T}{N}\right)  \geq e^{A^{\ast}\left(
\gamma_{0}\right)  \left(  \frac{T}{N}-t\right)  }\mathcal{G}\left(
\gamma_{0},t\right)  e^{A\left(  \gamma_{0}\right)  \left(  \frac{T}%
{N}-t\right)  },
\]
it follows that
\[
0\leq\left(  I-\mathcal{G}\left(  \gamma_{0},t\right)  e^{A^{\ast}\left(
\gamma_{0}\right)  \left(  \frac{T}{N}-t\right)  }\mathcal{G}^{-1}\left(
\gamma_{0},\frac{T}{N}\right)  e^{A\left(  \gamma_{0}\right)  \left(  \frac
{T}{N}-t\right)  }\right)  \leq I
\]
which implies $\left\vert X_{t}^{x_{0},u}\right\vert \leq e^{a_{0}t}\left\vert
x_{0}\right\vert ,$ for all $t\leq\frac{T}{N}\wedge T_{1}.$ Here, $a_{0}$
stands for an upper bound of the norm of linear operators $A\left(
\gamma\right)  $. Repeating this argument, it follows that
\[
\left\vert X_{t}^{x_{0},u}\right\vert \leq e^{a_{0}t}\left\vert x_{0}%
\right\vert ,
\]
for all $t\geq0$. We deduce that%
\begin{align*}
&  \mathbb{E}\left[  \left\vert X_{T}^{x_{0},u}\right\vert ^{2}\right] \\
&  \leq e^{2a_{0}T}\left\vert x_{0}\right\vert ^{2}\mathbb{P}\left(  \left(
T_{1}<\frac{T}{N}\right)  \cap\left(  T_{2}-T_{1}<\frac{T}{N}\right)
\cap...\cap\left(  T_{N}-T_{N-1}<\frac{T}{N}\right)  \right) \\
&  \leq e^{2a_{0}T}\left\vert x_{0}\right\vert ^{2}\left(  1-e^{-\frac{c_{0}%
T}{N}}\right)  ,
\end{align*}
where $c_{0}$ is an upper bound for the jump intensities $\lambda\left(
\gamma\right)  $. Since $N>0$ is arbitrary, it follows that our system is
approximately null-controllable and our proof is complete.
\end{proof}

Finally, we note that, in this particular setting, the Riccati equation would
be a true, backward one and have the form
\[
\left\{
\begin{array}
[c]{l}%
dK_{t}^{N}\\
=\left(  K_{t}^{N}A^{\ast}\left(  \gamma_{t}\right)  +A\left(  \gamma
_{t}\right)  K_{t}^{N}-N\Pi_{\left[  Ker\left(  B^{\ast}\left(  \gamma
_{t}\right)  \right)  \right]  ^{\bot}}\right)  dt+\int_{E}H_{t}^{N}\left(
\theta\right)  q\left(  dt,d\theta\right) \\
+\int_{E}\lambda\left(  \gamma_{t}\right)  H_{t}^{N}\left(  \theta\right)
^{\ast}\left(  I+K_{t}^{N}+H_{t}^{N}\left(  \theta\right)  \right)  ^{-1}%
H_{t}^{N}\left(  \theta\right)  Q\left(  \gamma_{t},d\theta\right)  dt,\text{
}t\in\left[  0,T\right]  .\\
K_{T}^{N}=0,\text{ }I+K_{t}^{N}+H_{t}^{N}\left(  \theta\right)  >0,\text{ for
almost all }t\in\left[  0,T\right]  .
\end{array}
\right.
\]
Unfortunately, the same approach as in the previous section fails to work
since we are unaware of a reference for the results on this type of equation.

\section{Conclusions and Perspectives\label{SectionConclusionsPerspectives}}

In this paper we have studied the approximate null-controllability property
for a class of jump Markov piecewise linear switched processes. These models
are inspired by lytic cycles in biological systems and prediction problems in
mathematical finance (energy markets). We have exhibited several necessary
conditions and established specific frameworks of sufficiency of these
algebraic criteria (independent coefficients or continuous jump mechanism).
Counterexamples are provided in the general case establishing the independence
of these conditions. We have tried to emphasize the stochastic characteristics
of our problem. We have shown that, as soon as jumps occur, the
controllability cannot be inferred from the analogous property of the related
deterministic systems. For the general framework, we have also proposed a
sufficiency criterion again based on algebraic invariance. Although very close
to the necessary criteria, this sufficient condition can be weakened in the
case of constant coefficients (in Example \ref{CtrlnotSuf1}). Our approach
shows that weaker invariance is sufficient as soon as local viability with
respect to the forward system can be obtained.

In order to relax the sufficient condition (following the method described for
constant coefficients in Subsection \ref{SubsectionConstantCoeff}), one would
try to characterize the local viability kernel of $Ker\left[  \left(
B^{0}\left(  \gamma_{0}\right)  \right)  ^{\ast}\right]  $. To this purpose,
one would be led to consider the backward stochastic Riccati equation%
\[
\left\{
\begin{array}
[c]{l}%
dK_{t}^{N}\\
=\left(  K_{t}^{N}A^{\ast}\left(  \gamma_{t}\right)  +A\left(  \gamma
_{t}\right)  K_{t}^{N}-N\Pi_{\left[  Ker\left(  B^{\ast}\left(  \gamma
_{t}\right)  \right)  \right]  ^{\bot}}\right)  dt+\int_{E}H_{t}^{N}\left(
\theta\right)  q\left(  dt,d\theta\right) \\
+\int_{E}\lambda\left(  \gamma_{t}\right)  \left(  f_{t}^{N}\left(
\theta\right)  \right)  ^{\ast}\left(  I+K_{t}^{N}+H_{t}^{N}\left(
\theta\right)  \right)  ^{-1}f_{t}^{N}\left(  \theta\right)  Q\left(
\gamma_{t},d\theta\right)  dt,\\
\text{where }f_{t}^{N}\left(  \theta\right)  :=\left(  C\left(  \theta\right)
K_{t}^{N}-H_{t}^{N}\left(  \theta\right)  \right)  ,\text{ }t\in\left[
0,T\right]  .\\
K_{T}^{N}=0,\text{ }I+K_{t}^{N}+H_{t}^{N}\left(  \theta\right)  >0,\text{ for
almost all }t\in\left[  0,T\right]  .
\end{array}
\right.
\]
Of course, the existence of a unique solution to this equation in the general
case is quite a difficult problem. For particular problems, we hope to be able
to give a solution by aggregating deterministic solutions (following the
approach for less restrictive one-dimensional BSDE given in
\cite{Confortola_Fuhrman_Jacod_2014}). Since all our counterexamples concern
bimodal systems, we are currently working on algebraic conditions which are
intermediate between the necessary and sufficient criteria proposed in the
present paper. This class is also relevant for the link between approximate
and approximate null-controllability. Finally, future work will concern
branching piecewise linear mechanisms to respond to the intuitive prediction
method in energy markets.

\textbf{Acknowledgement.} The work of the first author has been partially
supported by he French National Research Agency project PIECE, number
\textbf{ANR-12-JS01-0006.}

\bibliographystyle{plain}
\bibliography{bibliografie_06092014}

\end{document}